\documentclass[10pt]{article}

\usepackage[utf8]{inputenc}

\usepackage{url}

\usepackage[hidelinks,breaklinks]{hyperref}

\usepackage{geometry}
\usepackage[english]{babel}
\usepackage[backend=biber,style=alphabetic]{biblatex}
\usepackage{mathtools}
\usepackage{amsmath}
\usepackage{amssymb}
\usepackage{mathrsfs}
\usepackage{dsfont}
\usepackage{graphicx}
\usepackage{booktabs}
\usepackage{array}
\usepackage{paralist}
\usepackage{verbatim}
\usepackage{subfig}
\usepackage{tikz}
\usepackage{tikz-cd}
\usetikzlibrary{arrows}
\tikzset{node distance=1.8cm, auto}
\usepackage{csquotes}

\usepackage{lmodern}
\usepackage{sectsty}
\usepackage{amsthm}
\theoremstyle{definition}
\allsectionsfont{\sffamily\mdseries\upshape} 
\newtheorem{theorem}{Theorem}[section]
\newtheorem{proposition}[theorem]{Proposition}
\newtheorem{corollary}[theorem]{Corollary}

\newtheorem{example}[theorem]{Example}

\newtheorem{lemma}[theorem]{Lemma}
\newtheorem{remark}[theorem]{Remark}

\newtheorem{notation}[theorem]{Notation}

\usepackage{rotating}
\usepackage{graphicx}

\let\phi\varphi
\let\theta\vartheta
\DeclareMathOperator\Spec{Spec}

\DeclareMathOperator\Aut{Aut}

\renewcommand\char{\text{char}}

\DeclareMathOperator\Tr{Tr}

\DeclareMathOperator\Sym{Sym}

\let\tl\widetilde

\newcommand{\T}{\mathbf{T}}
\renewcommand{\S}{\operatorname{S}}

\newcommand{\bF}{\mathbb{F}}

\newcommand{\V}{\mathbb{V}}
\newcommand{\F}{\mathbb{F}}

\newcommand{\f}{\mathfrak{f}}
\renewcommand{\l}{\mathfrak{l}}

\title{\textsc{Periodicity of traces of Hecke operators modulo prime powers}}

\author{Jonas Bergstr\"om and Sjoerd de Vries}
\date{}

\begin{document}
\maketitle
\begin{abstract}
    \noindent
   We study traces of Hecke operators on spaces of elliptic cusp forms and Drinfeld cusp forms and show that, modulo any prime power, these traces are periodic in the weight. 
\end{abstract}

\tableofcontents

\section{Introduction}
Congruences between modular forms modulo prime powers play a central role in the study of arithmetic properties of modular forms and their associated Galois representations. While classical results focus on congruences between Fourier coefficients or Hecke eigenvalues of individual eigenforms, it is natural to also study congruences between traces of Hecke operators on different spaces of cusp forms. 

In this paper, we prove congruences, modulo prime powers, between traces of Hecke operators acting on spaces of classical cusp forms of different weights. 
Our result extends a result of Serre \cite{serreletter}, who determined such congruences modulo primes. 
The congruences hold for very general congruence subgroups; in particular, we do not assume that $p$-adic analytic methods apply, but nevertheless obtain results in the spirit of the Gouv\^ea--Mazur conjecture.

Our main tool is the trace formula for Hecke operators established in Deligne’s foundational work \cite{Deligne}. Our methods extend naturally to the setting of Drinfeld modular forms over function fields, giving a parallel theory of Hecke trace congruences in positive characteristic.

\subsection{The main theorems}

Our first main result is the following theorem about elliptic modular forms. For the notation, see further in Section~\ref{sec:mod_curves}. 
\begin{theorem}\label{thm:trace_periodicity} 
    Let $N\geq 1$ and let $H$ be a subgroup
    of $\mathrm{GL}_2(\mathbb Z/N\mathbb Z)$.
    Let $\ell$ be a prime and if $H$ is not representable then let $\ell \geq 5$. Let $q=p^a$ be a prime power with $\gcd(q,N)=1$. 
    Let $s \geq 1$, and  
    \[
    k_0 = \begin{cases} s-1 & \text{if } \ell \neq p; \\ s & \text{if } \ell = p \neq q; \\
    2s-1 & \text{if } \ell = p = q.
    \end{cases}
    \]
    \noindent If $k \geq k_0$ then   
    \[
    \Tr(F_q \, | \, S[H,k+2]) + \epsilon_k \equiv_{\ell^s} \Tr(F_q \, | \, S[H,k+2+n]),
    \]
    where 
    \[
    n =
    \begin{cases}
    \ell^{s-1}(\ell^2-1) & \text{if } \ell\geq 3 \text{ and }  \left( \frac{q}{\ell} \right) = -1 ; \\
    \ell^{s}(\ell^2-1)/2 & \text{if } \ell\geq 3 \text{ and } \left( \frac{q}{\ell} \right) = 1;\\
    \ell^{s}(\ell^2-1) & \text{if } \ell=2 \text{ and } q \text{ is a non-square mod } \ell^s; \\
    \text{lcm}(2,\ell^{s-1}(\ell^2-1)) & \text{if } \ell=2 \text{ and } q \text{ is an odd square mod } \ell^s; \\
    \ell^{s-1}(\ell-1) & \text{if } \ell = p, 
    \end{cases}
    \]
     and where $\epsilon_k=0$ for any $k>0$ and $\epsilon_0=-(q+1) \Tr(F_q \, | \, H_c^0(\mathcal X_H\otimes \overline{\F}_q))$ with
     $\mathcal X_H$ the compact modular curve with level structure $H$.  

    If $H$ is not representable and $\ell=2$ or $3$, then the above statement holds but with $s$ replaced by $s+\nu_\ell(H,q)$ in the definition of $n$. 
    In the case $\ell=2$ and $-\mathrm{Id}\in H$ then $s$ can instead be replaced by $s+\nu_2(H,q)-1$ in the definition of~$n$.
\end{theorem}
\begin{remark} 
    In a letter to Tate (see Equation (2) and (5) of \cite{serreletter}, but note the differences in notation), Serre proved that if $w_k=\Tr(F_p \,|\, S[\Gamma,k])-\Tr(F_p \,|\, S[\Gamma,k-(\ell-1)])$, then $w_{k+\ell+1} \equiv_{\ell} p \cdot w_k$ for all $\Gamma=\Gamma(N)$ with $N \geq 3$, $q=p \neq \ell$, $s=1$ and $k>\ell+1$. Under these assumptions it follows that 
    \[w_{k+\ell^2-1}  \equiv_{\ell} p^{\ell-1} \cdot w_k \equiv_{\ell} w_k
    \]
    and then that
   \begin{multline*}
   \Tr(F_p \,|\, S[\Gamma,k+\ell(\ell^2-1)])-\Tr(F_p\,|\, S[\Gamma,k])\equiv_{\ell} \sum_{i=0}^{\ell} \sum_{j=0}^{\ell-1} w_{k+(\ell-1)((\ell+1)(\ell-j)-i)} \\ \equiv_{\ell} \sum_{i=0}^{\ell} \sum_{j=0}^{\ell-1} w_{k+(\ell-1)(\ell+1-i)}=\sum_{i=0}^{\ell}\ell \cdot w_{k+(\ell-1)(\ell+1-i)}\equiv_{\ell} 0.
   \end{multline*}
    If $p$ is a square mod $\ell \neq 2$, then $w_{k+(\ell^2-1)/2}  \equiv_{\ell} w_k$ and
    \begin{multline*}
   \Tr(F_p \,|\, S[\Gamma,k+\ell(\ell^2-1)/2])-\Tr(F_p\,|\, S[\Gamma,k])\equiv_{\ell} \sum_{i=0}^{(\ell-1)/2} \sum_{j=0}^{\ell-1} w_{k+(\ell-1)((\ell+1)(\ell-j)/2-i)} \\ \equiv_{\ell} \sum_{i=0}^{(\ell-1)/2} \sum_{j=0}^{\ell-1} w_{k+(\ell-1)((\ell+1)/2-i)}=\sum_{i=0}^{(\ell-1)/2}\ell \cdot w_{k+(\ell-1)((\ell+1)/2-i)}\equiv_{\ell} 0.
   \end{multline*}
    This shows periodicities as in Theorem~\ref{thm:trace_periodicity}, but with a slightly larger period if $p$ is a non-square mod~$\ell$. 
\end{remark}

\begin{remark} \label{rmk:padic}
    Let $p \geq 5$. Theorem \ref{thm:trace_periodicity} in the case $\ell=p$ was previously proved for level $\mathrm{SL}_2(\mathbb Z)$ in \cite{koike,adolphson}. At level $\Gamma_1(N)$ with $p \nmid N$, it is related to the theory of $p$-adic modular forms. Let $P_k(t) = \det(1-U(p)t \, | \, M_k(N;r))$ denote the characteristic power series of the Atkin $U(p)$-operator acting on $r$-overconvergent $p$-adic modular forms of weight $k$ and level $\Gamma_1(N)$, where $r \in \mathbb Q_p$ satisfies $0 < \text{ord}_p(r) < p/(p+1)$. 
    Gouv\^ea and Mazur showed in \cite[Theorem~1]{GM} that for any $s \geq 1$ and $k \in \mathbb Z$, one has the congruence $P_k(t) \equiv_{p^s} P_{k+p^{s-1}(p-1)}(t)$. Theorem~\ref{thm:trace_periodicity} gives an analogous result for $T(p)$ acting on spaces of classical cusp forms of essentially arbitrary level, when $k$ is large enough. In the case of $\Gamma_1(N)$ this can be seen as a statement about $U(p)$ restricted to the $p$-old classical cusp forms.
    
    See also Proposition~\ref{prop:pols} for a statement about characteristic polynomials mod~$p$.
\end{remark}

Our second main result is the following version for Drinfeld modular forms. For the notation see Section~\ref{sec:drinf_modcurves}.

\begin{theorem}\label{thm:trace_periodicity_ff}
    Let $\mathcal{K}$ be a compact open subgroup of~$\mathrm{GL}_2(\hat{A})$ of minimal conductor~$\mathfrak n$. Let $\mathfrak p \nmid \mathfrak n$ and $\mathfrak l$ be maximal ideals of~$A$ and fix $n, s \geq 1$. Assume that $\mathfrak p^n$ is principal with monic generator~$\wp_n$. Write $\tilde{s} = \lceil \log_p(s) \rceil$. Let 
    \[
    k_0 = \begin{cases} s-1 & \text{if } \l \neq \mathfrak p; \\ s & \text{if } \l = \mathfrak p \neq \mathfrak p^n; \\
    2s-1 & \text{if } \l = \mathfrak p = \mathfrak p^n.
    \end{cases}
    \]
    Then for any $l \in \mathbb Z$ and any $k \geq k_0$, we have
    \[
    \Tr(\T_{\mathfrak p}^n \, | \, \S_{k+2,l}(\mathcal{K})) \equiv_{\l^s} \Tr(\T_{\mathfrak p}^n \, | \, \S_{k+2+m,l}(\mathcal{K})),
    \]
    where
    \[
    m = \begin{cases}
        p^{\tilde{s}}(|\l|^2-1) & \text{if } \l \neq \mathfrak p, \ p>2, \ \deg(\l) \equiv_2 0 \text{ and } \left( \frac{\wp_n}{\l} \right) = -1;\\
        p^{\tilde{s}+1}(|\l|^2-1)/2 & \text{if }\l \neq \mathfrak p, \ p>2,\ \deg(\l) \equiv_2 0 \text{ and } \left( \frac{\wp_n}{\l} \right) = 1;\\
        p^{\tilde{s}+1}(|\l|^2-1) & \text{if } \l \neq \mathfrak p \text{ and } (p = 2 \text{ or } \deg(\l) \equiv_2 1);\\
        p^{\tilde{s}}(|\l|-1) & \text{if } \l = \mathfrak p.
    \end{cases}
    \]
\end{theorem}

\section{Trace formulae}\label{sec:formulae}
\subsection{Modular curves}\label{sec:mod_curves}

Let $N\geq 1$ and let $\Gamma \subset \mathrm{SL}_2(\mathbb Z)$ be a congruence subgroup of level~$N$.
For any $k\geq 2$, let $S_{k}(\Gamma)$ denote the $\mathbb C$-vector space of cusp forms of weight $k$ with respect to $\Gamma$ and put $s_{k}=\dim S_{k}(\Gamma)$. For any Hecke eigenform $f \in S_k(\Gamma)$, let $a_p(f)$ denote the eigenvalue of the Hecke operator $T(p)$ acting on~$f$. There is then a $2$-dimensional $\ell'$-adic Galois representation corresponding to $f$, unramified for all $p \nmid N\ell'$, and whose characteristic polynomial of Frobenius $F_p$ equals $X^2-a_p(f)X+p^{k-1}$. 

For $k \geq 2$, let $S[\Gamma,k]$ denote the direct sum of the $\ell'$-adic Galois representations of a basis of Hecke eigenforms of $S_{k}(\Gamma)$.
For all $k\geq 2$ and all prime powers $q$, the trace of Frobenius $F_q$ acting on $S[\Gamma,k]$ will be an integer. 

Let now $H \subseteq \mathrm{GL}_2(\mathbb Z/N\mathbb Z)$ be a subgroup. 
Let $\mathcal Y_{H}$ (respectively $\mathcal X_{H}$) be the moduli space of elliptic curves (respectively generalized elliptic curves) with a level $H$-structure. This is a smooth (respectively smooth and proper) Deligne--Mumford stack defined over $\mathbb Z[1/N]$ with geometrically irreducible fibers over $\mathbb Z[\zeta_N]^{\det H}$, see \cite[Section IV.3 and IV.5]{DeligneRapoport}. In particular, $\mathcal Y_{H}(\mathbb C) \cong \coprod_{i=1}^d \Gamma_i \backslash \mathcal H$, where $\mathcal H$ is the complex upper half plane and each $\Gamma_i$ is a congruence subgroup of~$\mathrm{SL}_2(\mathbb Z)$. Write $S_k(H) = \bigoplus_{i=1}^d S_k(\Gamma_i)$ and $S[H,k] = \bigoplus_{i=1}^d S[\Gamma_i,k]$. 
When $\det H = (\mathbb Z/N\mathbb Z)^\times$, we have $S[H,k] = S[\Gamma,k]$ where $\Gamma = \pi^{-1}(H)$ for $\pi: \mathrm{SL}_2(\mathbb Z) \to \mathrm{GL}_2(\mathbb Z/N\mathbb Z)$. In particular, for $\Gamma \in \{\Gamma(N), \Gamma_1(N), \Gamma_0(N)\}$, one can find $H_\Gamma \subset \mathrm{GL}_2(\mathbb Z/N\mathbb Z)$ such that $S[H_\Gamma,k] = S[\Gamma,k]$. 

There is an universal elliptic curve $\psi: \mathcal E_H \to \mathcal Y_{H}$ (respectively $\psi':\mathcal E'_H \to \mathcal X_{H}$). Let $j: \mathcal Y_{H} \xhookrightarrow{} \mathcal X_{H}$ denote the natural inclusion.  For a field $k$ with $\char(k) \nmid N$, let $[\mathcal Y_{H}(k)]$ (or $[\mathcal X_{H}(k)]$, $[(\mathcal X_{H} \setminus \mathcal Y_{H})(k)]$) denote the set of $k$-isomorphism classes of objects of $\mathcal Y_{H}(k)$ (respectively $\mathcal X_{H}(k)$, $(\mathcal X_{H} \setminus \mathcal Y_{H})(k)$). So, a $y \in \mathcal Y_{H}(k)$ corresponds to an elliptic curve $E=(\mathcal E_H)_y$ over $k$ together with a level $H$-structure $\phi$ on $E$ over $k$. 

If, for all fields $k$ with $\char(k) \nmid N$ and all $y=(E,\phi) \in \mathcal Y_{H}(k)$, the automorphism group $\Aut_{k}(y)$ is trivial, then we call $H$ \emph{representable}. For instance, if 
$H \subseteq H_{\Gamma(N)}$ with $N \geq 3$ then $H$ is representable, and if $6 \mid N$ then $H$ is representable if and only if $\pi^{-1}(H)$ is torsion-free \cite[Th\'eor\`emes IV.2.7, VI.2.7]{DeligneRapoport}. 

If $H$ is not representable, $\ell$ is a prime and $q$ is any prime power, then put
\[
\nu_{\ell}(H,q):=\mathrm{max}\{\nu_{\ell}\bigl(\# \Aut_{\F_q}(y)\bigr) \ | \ y\in \mathcal Y_{H}(\F_q)\},
\]
where $\nu_{\ell}$ is the $\ell$-adic valuation on $\mathbb Z$ normalized so that $\nu_{\ell}(\ell)=1$. Note that $\nu_{2}(H,q) \leq 3$, $\nu_{3}(H,q) \leq 1$, and $\nu_{\ell}(H,q)=0$ for $\ell \geq 5$, see \cite[Theorem~10.1]{silverman}. The elements of $[(\mathcal X_{H} \setminus \mathcal Y_{H})(k)]$ are called \emph{cusps}. 
Finally define the $\ell'$-adic local system $\V_k=\mathrm{Sym}^k(R^1\psi_*\mathbb Q_{\ell'})$ on $ \mathcal Y_{H}$.

For an elliptic curve defined over $\F_q$, let $F_q$ denote the geometric Frobenius morphism acting on the $\ell'$-adic compactly supported cohomology group $H_c^1(E\otimes \overline{\F}_q,\mathbb Q_{\ell'})$, where $\ell' \nmid q$. Put $a_1(E)=\Tr(F_q \, | \,H_c^1(E\otimes \overline{\F}_q,\mathbb Q_{\ell'}))$ and note that $a_1(E)$ is an integer.

Finally, put
$\epsilon_k=0$ for any $k>0$ and $\epsilon_0=-(q+1) \Tr(F_q \, | \, H_c^0(\mathcal X_H\otimes \overline{\F}_q,\mathbb Q_{\ell'}))$.
If $\det H=(\mathbb Z/N\mathbb Z)^\times$ then $\epsilon_0=-(q+1)$.

\begin{theorem} \label{theorem-trace-Sk}
    Let $N\geq 1$ and let $H$ be a subgroup of $\mathrm{GL}_2(\mathbb Z/N\mathbb Z)$.
    
    For any $k \geq 0$ and prime power $q$ such that $\gcd(q,N)=1$, 
    \[
    \Tr(F_q \, | \, S[H,k+2]) +\epsilon_k = -\mathrm{eis}_{H,k}(q) - \sum_{j=0}^{\lfloor k/2 \rfloor} \binom{k-j}{j} (-q)^{j}[a_1^{k-2j}],
    \]
    where
    \[
    [a_1^{k-2j}] = \sum_{[(E,\phi)] \in \mathcal [\mathcal{Y}_{H}(\bF_q)]} \frac{a_1(E)^{k-2j}}{\# \Aut_{\bF_q}(E,\phi)} \in \mathbb Z,
    \]
    and $\mathrm{eis}_{H,k}(q)=\mathrm{eis}_{H,k+2i}(q) \in \mathbb Z$ for any $i \geq 0$. 
\end{theorem}
\begin{proof}
Fix any $\ell'\nmid q$. Below, we consider $\mathcal X_{H}$ over $\overline \F_q$ and we let $F_q$ denote the geometric Frobenius morphism. 
The Eichler--Shimura--Deligne relation together with the Grothendieck--Lefschetz trace formula, see \cite{Deligne} and \cite[Proposition~4.4]{Scholl}, says that for $k>0$,
\[
 \Tr(F_q \, | \, S[H,k+2]) =\Tr(F_q \, | \,H^1_c(\mathcal X_{H},j_*\V_k))=-\mathrm{eis}_{H,k}(q) - \sum_{y \in [\mathcal Y_{H}(\bF_q)]} \frac{\Tr(F_q \, | \,(\V_k)_{y})}{\# \Aut_{\bF_q}(y)},
\]
where 
\[  \mathrm{eis}_{H,k}(q)=\sum_{x \in [(\mathcal X_{H}\setminus\mathcal Y_{H})(\bF_q)]} \frac{\Tr(F_q \, | \,(\V_k)_{x})}{\# \Aut_{\bF_q}(x)}.
\]
Again using the Eichler--Shimura--Deligne relation and the Grothendieck--Lefschetz trace formula,
\begin{multline*}
 \Tr(F_q \, | \, S[H,2])=\Tr(F_q \, | \,H^1_c(\mathcal X_{H},\mathbb Q_{\ell'}))=\\
 =(q+1)\Tr(F_q \, | \,H^0_c(\mathcal X_{H},\mathbb Q_{\ell'}))-\mathrm{eis}_{H,0}(q) - \sum_{y \in [\mathcal Y_{H}(\bF_q)]} \frac{1}{\# \Aut_{\bF_q}(y)}.
\end{multline*}

A $y \in [\mathcal Y_{H}(\bF_q)]$ is an $\F_q$-isomorphism class of an elliptic curve $E=(\mathcal E_H)_y$ over $\F_q$ together with a level $H$-structure $\phi$ on $E$ over $\F_q$. The stalk of $\mathbb V_k$ at $(E,\phi)$ is equal to the $k$-th symmetric power of the $\ell'$-adic compactly supported cohomology group $H_c^1(E,\mathbb Q_{\ell'})$. After base change to $\overline \F_q$, the action of $F_q$ on the latter has characteristic polynomial $X^2 - a_1(E)X+q$.
  
Together with Lemma~\ref{lemma-hformula} we get that 
\begin{multline*}
\sum_{y \in [\mathcal Y_H(\bF_q)]} \frac{\Tr(F_q \, | \,(\V_k)_{y})}{\# \Aut_{\bF_q}(y)}=\sum_{[(E,\phi)] \in [\mathcal Y_{H}(\bF_q)]} \frac{\Tr(F_q \, | \,\Sym^k(H_c^1(E \otimes \overline \F_q,\mathbb Q_{\ell'})))}{\# \Aut_{\bF_q}(E,\phi)}=\\ =\sum_{[(E,\phi)] \in [\mathcal Y_{H}(\bF_q)]} \sum_{j=0}^{\lfloor k/2 \rfloor} \binom{k-j}{j} (-q)^{j}\frac{a_1(E)^{k-2j}}{\# \Aut_{\bF_q}(E,\phi)}.    
\end{multline*}

In the same way, an $x \in [(\mathcal X_{H} \setminus \mathcal Y_{H})(\bF_q)]$ is an $\F_q$-isomorphism class of a (singular) generalized elliptic curve $E=(\mathcal E'_H)_x$ over $\F_q$ together with a level $H$-structure $\phi$ on $E$ over $\F_q$. We again let $a_1(E)$ equal the trace of $F_q$ on the one-dimensional vector space $H_c^1(E\otimes \overline{\F}_q,\mathbb Q_{\ell'})$. We have $a_1(E)=1$ or $-1$ depending upon whether $E$ is of split or non-split multiplicative type. It follows that 
\begin{multline*}
 \mathrm{eis}_{H,k}(q)=
\sum_{[(E,\phi)] \in [(\mathcal X_{H}\setminus\mathcal Y_{H})(\bF_q)]} \frac{\Tr(F_q \, | \,\Sym^k(H_c^1(E\otimes \overline \F_q,\mathbb Q_{\ell'})))}{\# \Aut_{\bF_q}(E,\phi)} = \\ =\sum_{[(E,\phi)] \in [(\mathcal X_{H}\setminus\mathcal Y_{H})(\bF_q)]} \frac{a_1(E)^{k}}{\# \Aut_{\bF_q}(E,\phi)}
\end{multline*}
only depends upon the parity of~$k$.
  \end{proof}

\begin{remark}
For $N=1$, there is one split cusp and one non-split cusp over $\mathbb F_q$, both with an $\mathbb F_q$-automorphism $-1$ of order $2$ (and the two cusps are twists by $-1$ of each other) giving $\mathrm{eis}_{H,k}(q)=(1^k+(-1)^k)/2.$
For computations of $\mathrm{eis}_{H,k}(q)$ in the case when 
$\pi^{-1}(H)=\Gamma_1(3)$ or $\Gamma_1(4)$, see \cite[Theorem~2]{Hoffmanetal}. 
\end{remark}

\begin{remark}\label{rmk:frob_vs_hecke}
It follows from the definition of $S[H,k]$ that the trace of $F_p$ on~$S[H,k]$ equals the trace of $T(p)$ on~$S_k(H)$.
Note however that if $q = p^n$ with $n > 1$, then the trace of $F_q$ equals neither the trace of $T(q)$ nor the trace of $T(p)^n$ in general. If $\{f_1,\ldots,f_d\}$ is a basis of eigenforms for $S_k(H)$ and $T(p)f_i = a_p(f_i)f_i$, then
    \[
    \Tr(F_{p^n} \, | \, S[H,k]) = \sum_{i=1}^d \beta^n_i + \bar{\beta}^n_i,
    \]
    where $(X-\beta_i)(X-\bar{\beta}_i) = X^2 - a_p(f_i)X + p^{k-1}$. On the other hand, $\Tr(T(p)^n \, | \, S_k(H)) = \sum_{i=1}^d a_p(f_i)^n = \sum_{i=1}^d (\beta_i + \bar{\beta}_i)^n$, so we do obtain for any $n \geq 1$ the congruence
    \[
    \Tr(F_{p^n} \, | \, S[H,k]) \equiv_{p^{k-1}} \Tr(T(p)^n \, | \, S_k(H)) .
    \]

If $p \nmid N$, the trace of $F_q$ on $S[\Gamma_0(N),k]$ can be expressed as the trace of $U(p)^n$ acting on the space of $p$-oldforms in $S_k(\Gamma_0(Np))$, where $U(p)$ denotes the Atkin $U(p)$-operator. This follows because $F_p$ and $U(p)$ have the same characteristic polynomials. The same holds for $\Gamma_1(N)$.
\end{remark}

\begin{lemma} \label{lemma-hformula} For any $k \geq 1$, let $h_k(x_1,x_2)$ denote the complete homogeneous symmetric polynomial of weight~$k$. Then 
$$
h_k(x_1,x_2)=
\sum_{j=0}^{\lfloor k/2 \rfloor} \binom{k-j}{j} (-x_1 x_2)^j \, (x_1+x_2)^{k-2j}. 
$$
\end{lemma}
\begin{proof} For any integer $m \geq 0$, define $\mu_m$ to be 0 if $m$ is odd, and $-(x_1x_2)^{m/2}$ otherwise. Then we see that
$$
h_k(x_1,x_2)= \mu_k + \sum_{n=0}^{\lfloor k/2 \rfloor} (x_1x_2)^n(x_1^{k-2n} + x_2^{k-2n})
$$ 
and that
\begin{multline*}
\sum_{j=0}^{\lfloor k/2 \rfloor} \binom{k-j}{j}(-x_1 x_2)^j \, (x_1 + x_2)^{k-2j}=\\
=\sum_{i=0}^{\lfloor k/2 \rfloor} (-1)^i \binom{\lfloor k/2 \rfloor}{i}\binom{k-i}{ \lfloor k/2 \rfloor} \mu_k + \sum_{n=0}^{\lfloor k/2 \rfloor} \biggl( \sum_{i=0}^{n}  (-1)^i  \binom{n}{i}\binom{k-i}{n} \biggr) \, (x_1x_2)^n(x_1^{k-2n} + x_2^{k-2n}).
\end{multline*}
Looking at the equality, with $k \geq n$, 
$$\frac{d^n}{dx^n} \Bigl(x^{k-n}(x-1)^n\Bigr)=\frac{d^n}{dx^n} \sum_{i=0}^n (-1)^i \binom{n}{i} x^{k-i},
$$
evaluated at $x=1$ we find that 
$$1=\sum_{i=0}^n (-1)^i \binom{n}{i}\binom{k-i}{n}, 
$$
from which the lemma follows. 
\end{proof}

\subsection{Drinfeld modular curves}\label{sec:drinf_modcurves}
Let $C$ be a smooth projective geometrically connected curve over $\mathbb F_q$ and fix a closed point $\infty \in |C|$. Then $A := \mathcal O_{C}(C \setminus \{\infty\})$ is a Dedekind domain, which in the theory of Drinfeld modular forms plays the role of $\mathbb Z$ in the theory of elliptic modular forms. We denote by $K := \mathrm{Frac}(A)$ the function field of~$C$.

The theory explained in Section~\ref{sec:mod_curves} has an analogue in this setting. Write $\hat{A} := \prod_{\mathfrak p \neq \infty} A_{\mathfrak p}$, and let $\mathcal{K} \subset \mathrm{GL}_2(\hat{A})$ be a compact open subgroup of minimal conductor~$\mathfrak n$.
Then there is a Drinfeld modular curve $\mathfrak M_{\mathcal K}$, which is an affine, smooth, tame Deligne--Mumford stack over $\Spec(A[\mathfrak n^{-1}])$, that parameterizes rank 2 Drinfeld modules with a level ${\mathcal K}$-structure \cite{bockle,devries_ramanujan}. For any non-zero prime ideal $\mathfrak p$ of $A[\mathfrak n^{-1}]$, write $\mathbb F_{\mathfrak p} := A/\mathfrak p$ and write $\mathbb F_{\mathfrak p^n}$ for the unique degree~$n$ extension of~$\mathbb F_{\mathfrak p}$. The base change $\mathfrak M_{{\mathcal K},\mathfrak p} := \mathfrak M_{{\mathcal K}} \otimes_{A[\mathfrak n^{-1}]} \mathbb F_{\mathfrak p}$ then parameterizes pairs $(\phi,\lambda)$ with $\phi$ a Drinfeld $A$-module of rank~2 and characteristic~$\mathfrak p$, and $\lambda$ a level ${\mathcal K}$-structure on~$\phi$.

Let $\S_{k+2,l}({\mathcal K})$ denote the space of Drinfeld modular cusp forms of weight $k+2$, type $l$, and level ${\mathcal K}$. There is then a decomposition
\[
\S_{k+2,l}(\mathcal K) = \bigoplus_{i=1}^d \S_{k+2,l}(\Gamma_i),
\]
where each $\Gamma_i \subset \mathrm{GL}_2(K)$ is an arithmetic subgroup \cite[Lemma 4.14]{bockle}.
Explicitly, we have $\Gamma_i = \mathrm{GL}_2(K) \cap g_i \mathcal K g_i^{-1}$, where $g_i$ varies over a set of representatives for $\mathrm{GL}_2(K) \backslash \mathrm{GL}_2(\mathbb A_f) / \mathcal K$.

The spaces $\S_{k+2,l}(\mathcal K)$ admit actions of Hecke operators $\T'_{\mathfrak p}$, see \cite[Section~6]{bockle}. For any $n \geq 1$, we have $\T'_{\mathfrak p^n} = (\T'_{\mathfrak p})^n$ (in contrast to Remark~\ref{rmk:frob_vs_hecke}), and B\"ockle--Eichler--Shimura theory yields the equality \cite[Remark~4.12]{devries_ramanujan}
\begin{equation}\label{eq:traceformula_dmf}
\Tr(\T'_{\mathfrak p^n} \, | \, \S_{k+2,l}({\mathcal K})) = -\sum_{[(\phi,\lambda)] \in \mathfrak M_{{\mathcal K},\mathfrak p}(\mathbb F_{\mathfrak p^n})} \frac{h_k(\pi_{\phi},\bar{\pi}_{\phi}) (\pi_{\phi}\bar{\pi}_{\phi})^{l-k-1}}{\# \Aut_{\mathbb F_{\mathfrak p^n}}(\phi,\lambda)},
\end{equation}
where $\pi_{\phi},\bar{\pi}_{\phi}$ denote the roots of the Frobenius polynomial of~$\phi$ (see \cite[Chapter~4.2]{papikian} or \cite{gek_finite}).

\begin{remark}\label{rmk:aut=-1}
    The trace formula takes values in the characteristic~$p$ ring~$A$, so in particular only the residue classes $\# \Aut_{\mathbb F_{\mathfrak p^n}}(\phi,\lambda) \pmod{p}$ play a role. The automorphism group of a Drinfeld module (without level structure) over $\bF_{\mathfrak p^n}$ can be identified with the group of units in a finite field of the same characteristic, which has cardinality $-1 \pmod{p}$.
\end{remark}

\subsubsection{Renormalizing the Hecke operators}

Assume from now on that $\deg(\infty) = 1$, so that the residue field $\mathbb F_\infty$ of $\infty$ is $\mathbb F_q$. Fix a uniformizer $\pi_\infty$ of the local ring $\mathcal O_\infty \subset K_\infty$ and let $\text{pr}:\mathcal O_\infty \to \mathbb F_\infty$ be the reduction modulo $\pi_\infty$. We then say that $a \in K^\times$ is \emph{monic} if $\text{pr}(\pi_\infty^{-v_\infty(a)}a) = 1$. By assumption, for any $a \neq 0$ there is a unique $b \in \mathbb F_q^\times$ such that $ba$ is monic.

Now assume $\mathfrak M_{\mathcal{K},\mathfrak p}(\mathbb F_{\mathfrak p^n}) \neq \emptyset$. Then the ideal $\mathfrak p^n$ is principal
\cite[Cor.~5.2]{gek_finite} and, by the above, has a unique monic generator, which we denote by~$\wp_n$. In this case, $\wp_{nm} = \wp_n^m$ for any $m \geq 1$.

We now rescale the Hecke operators as follows: in weight $k+2$ and type $l$, define
\[
\T_{\mathfrak p^n} := \wp_n^{k-l+1} \T'_{\mathfrak p^n}.
\]
The advantage of this is twofold:
\begin{enumerate}
    \item By definition of Drinfeld modular forms, if $l \equiv l' \pmod{q-1}$ then $\S_{k,l}({\mathcal K}) = \S_{k,l'}({\mathcal K})$ as vector spaces. This identification is Hecke-equivariant with the new normalization, cf.~\cite[Remarks 5.33 and~6.12]{bockle}.
    \item If $\phi \in \mathfrak M_{\mathcal{K},\mathfrak p}(\mathbb F_{\mathfrak p^n})$, then $\pi_\phi \bar{\pi}_\phi = b_{\phi}\wp_n$ for a unique $b_{\phi} \in \mathbb F_q^\times$. Thus \eqref{eq:traceformula_dmf} yields
    \begin{equation}\label{eq:traceformula_dmf_renormalised}
        \Tr(\T_{\mathfrak p}^n \, | \, \S_{k+2,l}({\mathcal K})) = -\sum_{[(\phi,\lambda)] \in \mathfrak M_{{\mathcal K},\mathfrak p}(\mathbb F_{\mathfrak p^n})} \frac{h_k(\pi_{\phi},\bar{\pi}_{\phi}) b_{\phi}^{l-k-1}}{\# \Aut_{\mathbb F_{\mathfrak p^n}}(\phi,\lambda)}.
    \end{equation}
\end{enumerate}

\begin{remark}
    Compared to the original definition of Hecke operators $\T_{\mathfrak p}^{\mathbb F_q[T]}$ by Goss and Gekeler for $A = \mathbb F_q[T]$, we now have 
    \[
    \T_{\mathfrak p} = \wp^{-1} \T_{\mathfrak p}^{\mathbb F_q[T]}.
    \]
\end{remark}

\begin{remark}
    If $\mathfrak p^n$ is not a principal ideal, then the fact 
    that there are no Drinfeld modules over $\mathbb F_{\mathfrak p^n}$ \cite[Cor.~5.2]{gek_finite} implies, by~\eqref{eq:traceformula_dmf_renormalised}, that the trace of $\T_{\mathfrak p}^n$ is zero on any space of Drinfeld modular forms. Therefore we will always assume that $\mathfrak p^n$ is principal.
\end{remark}

\begin{theorem}\label{thm:dmf_traceform_general}
    Let $\mathcal{K} \subseteq \mathrm{GL}_2(\hat{A})$ be a compact open subgroup of level~$\mathfrak n$. Let $\mathfrak p \nmid \mathfrak n$ be a maximal ideal of~$A$. Suppose $\mathfrak p^n$ is principal and write $\wp_n$ for its unique monic generator.
    For any $k \geq 0$ and $l \in \mathbb Z$, we have
    \[
    \Tr(\T^n_{\mathfrak p} \, | \, \S_{k+2,l}(\mathcal{K}) ) = -\sum_{j=0}^{\lfloor k/2 \rfloor} \binom{k - j}{j} (-\wp_n)^{j} [c_{k-2j,l-j}],
    \]
    where
    \[
    [c_{k,l}] = \sum_{[(\phi,\lambda)] \in [\mathfrak M_{\mathcal{K},\mathfrak p}(\bF_{{\mathfrak p}^n})]} \frac{a_{\phi}^{k} b_{\phi}^{l-k-1}}{\# \Aut_{\mathbb F_{\mathfrak p^n}}(\phi,\lambda)}
    \]
    and $X^2 - a_{\phi}X + b_{\phi}\wp_n$ is the Frobenius polynomial of~$\phi$.
\end{theorem}
\begin{proof}
This follows from \eqref{eq:traceformula_dmf_renormalised} and Lemma~\ref{lemma-hformula}.
\end{proof}

\section{Proof of the main theorem for elliptic modular forms}\label{sec:elliptic-proof}

\begin{notation}\label{notation1} The following notation will be used in this section.  
\begin{itemize}
\item Fix $N \geq 1$ and fix a subgroup $H$ of $\mathrm{GL}_2(\mathbb Z/N\mathbb Z)$.
\item Fix a prime power $q=p^a$ with $\gcd(q,N)=1$. 
\item 
Let $\ell$ be a prime number and $s \geq 1$. Put $m_{2,1}=1$ and if $(\ell,s)\neq (2,1)$, then denote by $2m_{\ell,s}$ the order of the cyclic group $(\mathbb Z/\ell^s\mathbb Z)^\times$, i.e., $m_{\ell,s} = \ell^{s-1}(\ell-1)/2$. 
    \item 
Let
\[
\mathcal N(\bF_q) := \{ [(E,\phi)] \in [\mathcal Y_H(\bF_q)] \, | \, a_1(E) \equiv_\ell 0 \}, \qquad \mathcal U(\bF_q) := [\mathcal Y_H(\bF_q)] \setminus \mathcal N(\bF_q).
\]
Note that $a_1(E)^s \equiv_{\ell^s} 0$ for all $[(E,\phi)] \in \mathcal N(\bF_q)$ and $a_1(E)^{2m_{\ell,s}} \equiv_{\ell^s} 1$ for all $[(E,\phi)] \in \mathcal U(\bF_q)$. 
    \item 
For any triple of integers $(r,m,k) \in \mathbb Z^3$ with $k\geq 0$ and $m \geq 1$, we define the integer  
\[
f_{r,m,k} := \sum_{\substack{0 \leq j \leq \lfloor k/2 \rfloor \\ j \equiv \lfloor k/2 \rfloor -r \pmod{m} }} \binom{k-j}{j} (-q)^{j}.
\]

    \item 
Let $\delta_k := 0$ if $k$ is even and $\delta_k := 1$ if $k$ is odd.
\end{itemize}
\end{notation}

\begin{remark} In the following, when we write equalities modulo $\ell^s$ between rational numbers whose denominator is coprime to $\ell$, these should be interpreted as equalities in the ring $\mathbb Z/\ell^s\mathbb Z$. 
\end{remark}

\begin{lemma}\label{lem:trace_formula_modulo}
    Let $\ell$ be a prime number and if $H$ is not representable then let $\ell \geq 5$. For any 
    $s \geq 1$ and $k \geq s-1$, we have 
    \begin{align*}
    \Tr(F_q \, | \, S[H,k+2])+\epsilon_k \equiv_{\ell^s} -\mathrm{eis}_{H,k}(q) - N(k+2) - U(k+2),
    \end{align*}
    where
    \begin{align}\label{eq:N(k+2)}
   N(k+2) &:= \sum_{[(E,\phi)] \in \mathcal N(\bF_q)} \sum_{j=0}^{\lfloor (s-1-\delta_k)/2\rfloor} \binom{(k+\delta_k)/2+j}{2j+\delta_k} (-q)^{(k-\delta_k)/2 - j} \frac{a_1(E)^{2j+\delta_k}}{\# \Aut_{\F_q}(E,\phi)}; \\
    \label{eq:U(k+2)}
    U(k+2) &= \sum_{[(E,\phi)] \in \mathcal U(\bF_q)} \sum_{r=0}^{m_{\ell,s}-1} f_{r,m_{\ell,s},k} \cdot \frac{a_1(E)^{2r+\delta_k}}{\# \Aut_{\F_q}(E,\phi)}.
    \end{align}
\end{lemma}

\begin{proof}
    If $\ell \geq 5$ then $\ell \nmid \# \Aut_{\F_q}(E)$ for any elliptic curve $E$ over $\mathbb F_q$, see \cite[Theorem 10.1]{silverman}. For any level structure $\phi$, $\Aut_{\F_q}(E,\phi)$ is a subgroup of $\Aut_{\F_q}(E)$, which is trivial if $H$ is representable. Hence, under our assumptions on $\ell$ we have that $\# \Aut_{\F_q}(E,\phi)$ is invertible mod $\ell^s$.

    With a change of variables we get
\[
 \sum_{j=0}^{\lfloor k/2 \rfloor} \binom{k-j}{j} (-q)^{j}[a_1^{k-2j}]=
 \sum_{j=0}^{(k-\delta_k)/2} \binom{(k+\delta_k)/2+j}{2j+\delta_k} (-q)^{(k-\delta_k)/2-j}[a_1^{2j+\delta_k}].
\]
    The result now follows
    from Theorem~\ref{theorem-trace-Sk} using
    Lagrange's theorem in $(\mathbb Z/\ell^s\mathbb Z)^\times$. 
\end{proof}

\begin{lemma}[Lucas's theorem for prime powers]\label{lem:lucas}
    For any prime $\ell$, $s \geq 1$ and $j \geq 0$, put $t = s + \lfloor \log_{\ell}(j)\rfloor$ if $j > 0$ and $t = 0$ if $j = 0$. Then for any $k \geq 0$, we have
    \[
    \binom{k}{j} \equiv \binom{k+\ell^t}{j} \pmod{\ell^s}.
    \]
\end{lemma}

\begin{proof}
    This follows immediately from \cite[Theorem~1]{prime-power-lucas}\footnote{In the statement of this theorem, $m = n + r$ should read $n = m+r$.}.
\end{proof}

\begin{notation} For any prime $\ell$ and $s \geq 1$, put 
    \[
     n_{N} := \begin{cases}
 2\ell^{s-1}(\ell-1) & \text{if } \ell \geq 2  \text{ and } q \text{ is a non-square mod }\ell^s; \\
    \text{lcm}(4,\ell^{s-1}(\ell-1)) & \text{if } \ell \geq 3  \text{ and } q \text{ is a square mod }\ell^s; \\  
     \text{lcm}(2,\ell^{s-1}(\ell-1)) & \text{if } \ell =2  \text{ and } q \text{ is a square mod }\ell^s.        
     \end{cases}
     \] 
\end{notation}

\begin{lemma} \label{lem:N_periodic}
Let $\ell$ be a prime number and if $H$ is not representable then let $\ell \geq 5$. Assume also that $\ell \nmid q$. 
    For any $s \geq 1$ and $k \geq s-1$, we have
    \[
    N(k+2) \equiv_{\ell^s} N(k+2+n_N).
    \]
\end{lemma}

\begin{proof} 
    Write $ N(k+2) = \sum_{j=0}^{\lfloor (s-1-\delta_k)/2\rfloor} N_j(k+2)$, where
   \[
    N_j(k+2) := 
    \binom{(k+\delta_k)/2+j}{2j+\delta_k} (-q)^{(k-\delta_k)/2 - j}  \sum_{[(E,\phi)] \in \mathcal N(\bF_q)}  \frac{a_1(E)^{2j+\delta_k}}{\# \Aut_{\F_q}(E,\phi)}.
    \]
    By definition of $\mathcal N(\mathbb F_q)$, we see that the sum over $\mathcal N(\mathbb F_q)$ is divisible by
    $\ell^{2j+\delta_k}$, so it suffices to show that 
    \[
     \binom{(k+\delta_k)/2+j}{2j+\delta_k} \equiv_{\ell^{s-2j-\delta_k}}  \binom{(k+\delta_k)/2+j+n_{N}/2}{2j+\delta_k}
     (-q)^{n_{N}/2} \quad \text{for all } 0 \leq j \leq \lfloor (s-1-\delta_k)/2\rfloor.
    \]
    Since $\ell \nmid q$, we have that $-q$ is a unit mod $\ell^{s}$ and $(-q)^{n_{N}/2}
    \equiv_{\ell^s} 1$ by definition of~$n_N$. The required congruence between the binomial coefficients follows from Lemma~\ref{lem:lucas}.
\end{proof}

\subsection{The case of $\ell \geq 3$ and $\ell \nmid q$}

\begin{notation} For any $s \geq 1$ and prime $\ell \geq 3$ such that $\ell \nmid q$ put 
    \[
    n_{U} := \begin{cases}
\ell^{s-1}(\ell^2-1) & \text{if } \left( \frac{q}{\ell} \right) = -1; \\
    \ell^{s}(\ell^2-1)/2 & \text{if } \left( \frac{q}{\ell} \right) = 1.        
    \end{cases}
    \]    
\end{notation}

It remains to show the periodicity of $U(k+2)$. This will follow 
from periodicity of the $f_{r,m,k}$, which we establish with the next lemmas.

\begin{lemma}\label{lem:periodic_criterion}
Let $R$ be a commutative ring with unity. Suppose we have 
\[
g(x) = \sum_{k \geq 0} a_k x^k = \frac{f(x)}{d(x)} \in R[\![x]\!] \cap R(x).
\]
Suppose $d(x)$ divides $x^n - 1$. If the leading coefficient of~$d(x)$ is not a zero-divisor, we have
\[
a_k = a_{k+n} \quad \text{for all } k > \deg(f) - \deg(d).
\]
\end{lemma}

\begin{proof}
    Write $d(x)h(x) = x^n-1$. Then $(x^n-1)g(x) = f(x)h(x)$. On the other hand,
    \[
    (x^n-1)g(x) = \sum_{k \geq -n} (a_k - a_{k+n})x^{k+n},
    \]
    where $a_k = 0$ if $k < 0$. In particular, if $k + n > \deg(f) + \deg(h)$, we have $a_{k} = a_{k+n}$. If the leading coefficient of $d$ is not a zero-divisor, then $\deg(d) + \deg(h) = n$, which yields the claim.
\end{proof}

\begin{lemma}\label{lem:rootofunity}
    Let $K$ be a field. Suppose that $\zeta \in K^\times$ is an element of order~$m < \infty$. Then for all $A, B \in \bar{K}$, we have
    \[
    \prod_{i=1}^m (A - \zeta^iB) = A^m - B^m.
    \]
\end{lemma}
\begin{proof}
    The proof is standard and therefore omitted. 
\end{proof}

\begin{lemma}\label{lem:series_rational} 
    Fix $\delta \in \{0,1\}$, $m \geq 1$ and $0 \leq r < m$. There is then a polynomial $h(x) \in \mathbb Z[x]$ of degree at most $4m-2-\delta$ such that
    \[
    \sum_{\substack{k \geq 0 \\ k\equiv_2 \delta}} f_{r,m,k} \cdot x^{k} = \frac{h(x)}{(1+qx^2)^{2m} - x^{2m}}. 
    \]
\end{lemma}

\begin{remark}
In Lemma~\ref{lem:series_rational}, $h(x)$ and $(1+qx^2)^{2m} - x^{2m}$ are not necessarily coprime.    
\end{remark}

\begin{proof}
Consider the generating series
\[
g(x,y) := \frac{1}{1-(1+y)x} = \sum_{n \geq 0}\sum_{j \geq 0} \binom{n}{j} y^j x^n.
\]
Setting $y = -qx$, we obtain
\[
h_0(q,x) := \frac{1}{1-x+qx^2} = \sum_{n \geq 0} \sum_{j \geq 0} \binom{n}{j} (-q)^j x^{n+j} = \sum_{k \geq 0}\left( \sum_{j=0}^{\lfloor k/2 \rfloor} \binom{k-j}{j} (-q)^j \right) x^k.
\]
Write $h_1(q,x) := (h_0(q,x) + (-1)^{\delta}h_0(q,-x))/2$, so that only the even or odd powers of $x$ occur. Let $\zeta$ be a primitive $2m$-th root of unity. Define
\[
H(x) := \frac{1}{m} \sum_{i=1}^m \zeta^{i(2r+\delta)} h_1(\zeta^{2i}q,\zeta^{-i} x).
\]
Then we have
\begin{align*}
H(x) &= \sum_{\substack{k \geq 0 \\ k\equiv_2 \delta}} \left( \sum_{j = 0}^{\lfloor k/2 \rfloor} \frac{1}{m} \sum_{i = 1}^m \zeta^{2i(j-\lfloor k/2 \rfloor+r)} \binom{k-j}{j} (-q)^j \right) x^{k} \\
&= \sum_{\substack{k \geq 0 \\ k\equiv_2 \delta}} \left( \sum_{\substack{0 \leq j \leq \lfloor k/2 \rfloor \\ j \equiv \lfloor k/2 \rfloor -r \pmod{m}}} \binom{k-j}{j} (-q)^j \right) x^{k} = \sum_{\substack{k \geq 0 \\ k\equiv_2 \delta}} f_{r,m,k} \cdot x^{k}.
\end{align*}
On the other hand, we can express $H(x)$ as a rational function by bringing the $h_1(\zeta^{2i}q,\zeta^{-i}x)$ under a common denominator. This denominator can be written as
\[
\prod_{i=1}^m (1+qx^2-\zeta^{-i}x)(1+qx^2+\zeta^{-i}x) = \prod_{i=1}^m ((1+qx^2)^2 - \zeta^{-2i}x^2) = (1+qx^2)^{2m} - x^{2m} \in \mathbb Z[x].
\]
Since $H(x) \in \mathbb Z[\![x]\!]$, the numerator lies in $\bar{\mathbb Q}[x] \cap \mathbb Z[\![x]\!] = \mathbb Z[x]$. Since each $h_1(\zeta^{2i}q,\zeta^{-i}x)$ has a numerator of degree~$2-\delta$ and a denominator of degree~4, we obtain the claimed statement.
\end{proof}

\begin{lemma}\label{lem:(f-1)^ell}
    Let $\ell$ be prime and $n \geq 1$. Then there exists $g(x) \in \mathbb Z[x]$ of degree $n(\ell-2)$ such that 
    \[
    (x^n-1)^{\ell} = x^{n\ell} - 1 + (x^{n}-1)\ell g.
    \]
\end{lemma}

\begin{proof}
    If $\ell=2$ then 
    \[
        (x^n-1)^{\ell} = x^{n\ell} - 1 + (x^{n}-1)(-\ell).
    \]
    
    If $\ell$ is odd, define $\tilde{g}(x) := \left( (x^n-1)^{\ell} - x^{n\ell} + 1 \right)/\ell$. Since $\ell \mid \binom{\ell}{i}$ for all $0 < i < \ell$, we have $\tilde{g} \in \mathbb Z[x]$. If $\zeta$ is an $n$-th root of unity, then $\tilde{g}(\zeta) = 0$, so $x^n-1$ divides $\tilde{g}$. 
\end{proof}

\begin{lemma}\label{lem:div_upgrade}
    Let $\ell$ be prime and $m,n \geq 1$. Let $f \in \mathbb Z[x]$ and assume that $f$ divides $x^n - 1$ modulo $\ell^m$. Then, for any $r \geq 1$, $f$ divides $x^{n \ell^r} - 1$ modulo $\ell^{m+r}$.
\end{lemma}

\begin{proof}
    By induction, it suffices to prove the statement with $r = 1$. By assumption, there exist polynomials $h_1, r_1 \in \mathbb Z[x]$ such that
    \[
    x^n-1 = fh_1 + \ell^{m} r_1.
    \]
    Raising both sides to the power $\ell$, we see that there exist polynomials $h_2, r_2 \in \mathbb Z[x]$ such that
    \[
    (x^n-1)^{\ell} = fh_2 + \ell^{m+1}r_2.
    \]
    By Lemma~\ref{lem:(f-1)^ell}, we also know that for some $g \in \mathbb Z[x]$, we have
    \[
    (x^n-1)^{\ell} = x^{n\ell} - 1 + (fh_1 + \ell^{m} r_1)\ell g
    \]
    and hence we conclude that
    \[
    x^{n\ell}-1 = f(h_2 - h_1 \ell g) + \ell^{m+1}(r_2 -r_1g). \qedhere
    \]
\end{proof}

\begin{lemma}\label{lem:denom_divisible}
     Let $\ell \geq 3$ be prime and assume $\ell \nmid q$.
     Let $s \geq 1$ and $1 \leq t \leq s$. Put $m = m_{\ell,t} = \ell^{t-1}(\ell-1)/2$. Then the polynomial $d_{q,t} := (1+qx^2)^{2m} - x^{2m}$ divides $x^{n_{U}}-1$ modulo $\ell^{s+1-t}$.
\end{lemma}

\begin{proof}
    Let $\alpha \in \mathbb Z$ be a primitive element modulo~$\ell$. Consider the polynomial
    \[
    D_q := \prod_{i=1}^{(\ell-1)/2} (1 + qx^2 - \alpha^i x)(1 + qx^2 + \alpha^i x).
    \]
    Since $\alpha^2$ has order $(\ell-1)/2$ modulo~$\ell$, Lemma~\ref{lem:rootofunity}  yields
    \[
    D_q \equiv (1 + qx^2)^{\ell-1} - x^{\ell-1} \pmod{\ell},
    \]
    and applying Frobenius $t-1$ times gives
    \[
    D_q^{\ell^{t-1}} \equiv d_{q,t} \pmod{\ell}.
    \]
    The polynomial $D_q \pmod{\ell}$ factors as a product of $f_i := 1 - \alpha^i x + qx^2$ for $1 \leq i \leq \ell-1$. Note that $f_i$ and $f_j$ have no roots in common if $i \neq j$.
    
    Assume first that $q$ is a non-square mod~$\ell$. Then none of the $f_i$ is a square. It follows that in this case, $D_q$ is a product of distinct irreducible polynomials of degree at most~2. Since $x^{\ell^2-1}-1 \equiv_{\ell} \prod_{P \in Q} P$, where $Q$ is the set of monic irreducible polynomials of degree $\leq 2$ in $\mathbb F_{\ell}[x]$, we see that $D_q$ divides $x^{\ell^2-1}-1$ mod~$\ell$ and hence $d_{q,t}$ divides $(x^{\ell^2-1}-1)^{\ell^{t-1}} = x^{\ell^{t-1}(\ell^2-1)}-1$ mod~$\ell$.
    We conclude by applying Lemma~\ref{lem:div_upgrade} with $r=s-t$.
    
    Next, assume that $q = r^2$ is a square mod~$\ell$. If we can show that $D_q$ divides $x^{\ell(\ell^2-1)/2}-1$ mod~$\ell$, the result follows in the same way as above. Note that $x^{(\ell^2-1)/2}-1 \equiv_{\ell} \prod_{P \in Q'} P$, where $Q'$ is the set of monic irreducible polynomials of degree $\leq 2$ in $\bF_{\ell}[x]$ whose roots are squares in~$\bF_{\ell^2}$. Thus, we aim to show that the roots of each $f_i$ are squares in~$\bF_{\ell^2}$. If $f_i$ has a root in~$\bF_{\ell}$, this is immediate, so assume $f_i$ is irreducible. It suffices to show that $x$ is a square in $\bF_{\ell^2} = \bF_{\ell}[x]/(f_i)$.

    Solving $x = (Ax + B)^2$ in $\bF_{\ell^2}$ for $A,B$ in $\bF_{\ell}$ yields
    \[
    B^2 - r^{-2}A^2 + x(2AB+r^{-2}\alpha^i A^2-1) = 0.
    \]
    Set $A = \pm rB$. Then the equation has a solution if and only if
    \begin{equation}\label{eq:B}
    B^2(\alpha^i \pm 2r) = 1
    \end{equation}
    has a solution for some $B \in \bF_{\ell}$. Since $f_i$ is irreducible, its discriminant $\alpha^{2i}-4q = (\alpha^i + 2r)(\alpha^i - 2r)$ is a non-square, so precisely one of $\alpha^i + 2r, \alpha^i - 2r$ is a square. Hence \eqref{eq:B} has a solution in $\bF_{\ell}$ and $x$ is a square, as desired.

    Finally, since $q$ is square, two of the $f_i$ are squares of linear factors. Therefore $D_q$ divides $(x^{(\ell^2-1)/2}-1)^{\ell} = x^{\ell(\ell^2-1)/2}-1$ mod~$\ell$, which is what we wanted to show.
\end{proof}

\begin{corollary}\label{cor:n_U}
    Let $\ell \geq 3$ be prime and assume $\ell \nmid q$.
     Let $s \geq 1$ and $1 \leq t \leq s$. 
     For any $k \geq 0$ and any $r \in \mathbb Z$, we have
    \[
    f_{r,m_{\ell,t},k} \equiv_{\ell^{s+1-t}} f_{r,m_{\ell,t},k+n_U}.
    \]
\end{corollary}

\begin{proof}
    By Lemmas \ref{lem:series_rational} and~\ref{lem:denom_divisible}, with $m = m_{\ell,t}$, 
    we can 
    apply Lemma~\ref{lem:periodic_criterion} with $R = \mathbb Z/\ell^{s+1-t}\mathbb Z$ and $a_k = f_{r,m_{\ell,t},k} \pmod{\ell^{s+1-t}}$ to conclude that $a_k =a_{k+n_U}$ in $R$ for any $k \geq 0$.
\end{proof}

\begin{lemma}\label{lem:f_sum}
    Let $\ell \geq 3$ be prime, $s \geq 1$ and $1 \leq t \leq s$. For any $k \geq 0$ and any $r \in \mathbb Z$, we have 
    \[
    \sum_{i=0}^{\ell^{s-t}-1} f_{r+im_{\ell,t},m_{\ell,s},k} = f_{r,m_{\ell,t},k}.
    \]
\end{lemma}

\begin{proof}
    From the definition of $f_{r,m,k}$, we need to show that every $j \equiv \lfloor k/2\rfloor - r \pmod{m_{\ell,t}}$ can be written as $\lfloor k/2 \rfloor - r - im_{\ell,t} \pmod{m_{\ell,s}}$ for a unique $0 \leq i < \ell^{s-t}$. This follows because $m_{\ell,s}/m_{\ell,t} = \ell^{s-t}$. 
\end{proof}

We can now deduce the periodicity of $U(k+2)$.

\begin{lemma}\label{lem:U_periodic}
    Let $\ell \geq 3$ be prime, $s \geq 1$ and suppose $\ell \nmid q$. 
    For any $k \geq 0$, we have
    \[
    U(k+2) \equiv_{\ell^s} U(k + 2 + n_U).
    \]
\end{lemma}

\begin{proof}
    For $j \geq 0$, define
    \[
    a_1|_{\mathcal U}^{j} := \sum_{[(E,\phi)] \in \mathcal U(\bF_q)} \frac{a_1(E)^{j}}{\# \Aut(E,\phi)}
    \]
    and write $a_1|_{\mathcal U}^{j} \equiv_{\ell^s} \sum_{t=1}^s b_{j,t} \ell^{t-1}$ with $0 \leq b_{j,t} < \ell$. Then it follows from the definition of $\mathcal U(\bF_q)$ and Lagrange's theorem that $b_{j,t} = b_{j+2m_{\ell,t},t}$. Thus we have
    \begin{align*}
    U(k+2) &= \sum_{r=0}^{m_{\ell,s}-1} f_{r,m_{\ell,s},k} \,a_1|_{\mathcal U}^{2r+\delta_k} \\
    &\equiv_{\ell^s} \sum_{t=1}^s \ell^{t-1} \sum_{r=0}^{m_{\ell,s}-1} f_{r,m_{\ell,s},k}\, b_{2r+\delta_k,t} \\
    &\equiv_{\ell^s} \sum_{t=1}^s \ell^{t-1} \sum_{r=0}^{m_{\ell,t}-1} \sum_{i=0}^{\ell^{s-t}-1} f_{r+im_{\ell,t},m_{\ell,s},k} \,b_{2(r+im_{\ell,t})+\delta_k,t} \\
    &\equiv_{\ell^s} \sum_{t=1}^s \ell^{t-1} \sum_{r=0}^{m_{\ell,t}-1} f_{r,m_{\ell,t},k} \,b_{2r+\delta_k,t},
    \end{align*}
    where we used Lemma \ref{lem:f_sum} to obtain the last equality. Since $n_U$ is even, we have $\delta_k = \delta_{k+n_U}$. Thus the above expression shows that $U(k+2) \equiv_{\ell^s} U(k+2+n_U)$ by Corollary~\ref{cor:n_U}.
\end{proof}

\begin{proof}[Proof of Theorem~\ref{thm:trace_periodicity} when $H$ is representable, $\ell \nmid q$ and $\ell \geq 3$.]

Recall that $\mathrm{eis}_{H,k}(q)=\mathrm{eis}_{H,k+2i}(q)$ for any $i \geq 0$, by Theorem~\ref{theorem-trace-Sk}. 
Together with Lemmas~\ref{lem:trace_formula_modulo}, \ref{lem:N_periodic} and \ref{lem:U_periodic}, this gives the period $\text{lcm}(2,n_N,n_U) = n_U$, where $n_N \mid n_U$ because $\ell^2-1 \equiv_8 0$.
\end{proof}  

\subsection{The case of \texorpdfstring{$\ell=2$ and $\ell \nmid q$}{elltwo}}\label{sec:ell=2}
Recall that $m_{2,1}=1$. 

\begin{notation} For $\ell=2$ and $s\geq 1$ put $n'_{U} := \text{lcm}(2,\ell^{s-1}(\ell^2-1))$. 
\end{notation}

\begin{lemma}\label{lem:denom_divisible_even}
     Let $\ell =2$ and suppose $\ell \nmid q$.
     Let $s \geq 1$ and $1 \leq t \leq s$. Put $m = \ell^{t-1}(\ell-1)/2$. Then the polynomial $d_{q,t} := (1+qx^2)^{2m} - x^{2m}$ divides $x^{n'_{U}}-1$ modulo $\ell^{s+1-t}$.
\end{lemma}
\begin{proof}
    We have that $d_{q,1}=1+qx^2-x$ divides $x^{\ell^2-1}-1$ mod $\ell$ and hence that $d_{q,t} \equiv_{\ell} d_{q,1}^{\ell^{t-1}}$ divides $(x^{\ell^2-1}-1)^{\ell^{t-1}} = x^{\ell^{t-1}(\ell^2-1)}-1$.  
    We conclude by applying Lemma~\ref{lem:div_upgrade} with $r=s-t$.
\end{proof}

\begin{corollary}\label{cor:n_U_even}
    Let $\ell =2$ and assume $\ell \nmid q$.
     Let $s \geq 1$ and $1 \leq t \leq s$. 
     For any $k \geq 0$ and any $r \in \mathbb Z$, we have
    \[
    f_{r,m_{\ell,t},k} \equiv_{\ell^{s+1-t}} f_{r,m_{\ell,t},k+n'_U}.
    \]
\end{corollary}

\begin{proof}
    If $t \geq 2$ we apply Lemma~\ref{lem:periodic_criterion} with $R = \mathbb Z/\ell^{s+1-t}\mathbb Z$ and $a_k = f_{r,m_{\ell,t},k} \pmod{\ell^{s+1-t}}$. The criteria of Lemma~\ref{lem:periodic_criterion} are satisfied with $n = \ell^{s-1}(\ell^2-1)$ for $t\geq 2$ by Lemma~\ref{lem:denom_divisible_even} together with Lemma \ref{lem:series_rational}.
    If $t=1$ then $m_{2,1}=1$ and we have   
    \begin{equation} \label{eq-m21}
    \sum_{k \geq 0} f_{0,1,k} \cdot x^{k} = \frac{1}{1+qx^2-x}.
    \end{equation}
    The denominator divides $x^{\ell^2-1}-1$ mod $\ell$, so applying Lemma~\ref{lem:div_upgrade} with $r = s-t$ lets us argue as in the case $t \geq 2$.
\end{proof}

\begin{lemma}\label{lem:f_sum_even}
    Let $\ell=2$, $s \geq 2$ and $2 \leq t \leq s$. For any $k \geq 0$ and any $r \in \mathbb Z$, we have 
    \[
    \sum_{i=0}^{\ell^{s-t}-1} f_{r+im_{\ell,t},m_{\ell,s},k} = f_{r,m_{\ell,t},k}.
    \]
    For $t=1$ we have 
    $\sum_{i=0}^{m_{\ell,s}-1} f_{r+im_{\ell,1},m_{\ell,s},k} = f_{0,1,k}. $
\end{lemma}
\begin{proof}
    If $t \geq 2$, this follows in the same way as Lemma~\ref{lem:f_sum}. For $t=1$, the definition of $f_{r,m,k}$ gives 
        \[
    \sum_{i=0}^{m_{\ell,s}-1} f_{r+im_{\ell,1},m_{\ell,s},k}=\sum_{i=0}^{m_{\ell,s}-1} f_{r+i,m_{\ell,s},k}=f_{0,1,k}.\qedhere
    \]
\end{proof}

\begin{lemma}\label{lem:U_periodic_even}
    Let $\ell=2$, $s \geq 1$ and suppose $\ell \nmid q$. For any $k \geq 0$, we have
    \[
    U(k+2) \equiv_{\ell^s} U(k + 2 + n'_U).
    \]
\end{lemma}

\begin{proof}
    For $j \geq 0$, define
    \[
    a_1|_{\mathcal U}^{j} := \sum_{[(E,\phi)] \in \mathcal U(\bF_q)} \frac{a_1(E)^{j}}{\# \Aut(E,\phi)}
    \]
    and write $a_1|_{\mathcal U}^{j} \equiv_{\ell^s} \sum_{t=1}^s b_{j,t} \ell^{t-1}$ with $0 \leq b_{j,t} < \ell$. Then it follows from the definition of $\mathcal U(\bF_q)$ and Lagrange's theorem that $b_{j,t} = b_{j+2m_{\ell,t},t}$. Thus we have
    \begin{align*}
    U(k+2) &= \sum_{r=0}^{m_{\ell,s}-1} f_{r,m_{\ell,s},k} \,a_1|_{\mathcal U}^{2r+\delta_k} \\
    &\equiv_{\ell^s} \sum_{t=1}^s \ell^{t-1} \sum_{r=0}^{m_{\ell,s}-1} f_{r,m_{\ell,s},k}\, b_{2r+\delta_k,t} \\
    &\equiv_{\ell^s} \sum_{r=0}^{m_{\ell,s}-1} f_{r,m_{\ell,s},k} \,b_{2r+\delta_k,1}+\sum_{t=2}^s \ell^{t-1} \sum_{r=0}^{m_{\ell,t}-1} \sum_{i=0}^{\ell^{s-t}-1} f_{r+im_{\ell,t},m_{\ell,s},k} \,b_{2(r+im_{\ell,t})+\delta_k,t} \\
    &\equiv_{\ell^s} f_{0,1,k} \,b_{\delta_k,1}+\sum_{t=2}^s \ell^{t-1} \sum_{r=0}^{m_{\ell,t}-1} f_{r,m_{\ell,t},k} \,b_{2r+\delta_k,t},
    \end{align*}
    where we used Lemma \ref{lem:f_sum_even} to obtain the last equality. Since $n'_U$ is even, we have $\delta_k = \delta_{k+n'_U}$. Thus the above expression shows that $U(k+2) \equiv_{\ell^s} U(k+2+n'_U)$ by Corollary \ref{cor:n_U_even}.
\end{proof}

\begin{proof}[Proof of Theorem~\ref{thm:trace_periodicity} when $H$ is representable, $\ell \nmid q$ and $\ell =2$.] 

Recall that $\mathrm{eis}_{H,k}(q)=\mathrm{eis}_{H,k+2i}(q)$ for any $i \geq 0$, by Theorem~\ref{theorem-trace-Sk}.
By Lemmas~\ref{lem:trace_formula_modulo}, \ref{lem:N_periodic} and~\ref{lem:U_periodic_even}, the period is $\text{lcm}(2,n_N,n'_{U})$, which gives the result. 
\end{proof}

\subsection{The case of \texorpdfstring{$\ell \mid q$}{ellmidq}}\label{sec:ell=p}

\begin{proof}[Proof of Theorem~\ref{thm:trace_periodicity} when $H$ is representable and $\ell \mid q$.]
    We argue as in the case $\ell \neq p$. Since $\ell \mid q$, we see from Equation~\eqref{eq:N(k+2)} that $N(k+2) \equiv_{\ell^s} 0$ for $k$ large enough; in fact, this holds for $k \geq 2s-\delta_k$ 
    if $a = 1$, and for $k \geq s$ if $a \geq 2$.
    
    For $U(k+2)$, set $1 \leq t \leq s$, $m = m_{\ell,t}$, and $n =
    \ell^{s-1}(\ell-1)$. We aim to show that
    \begin{equation}\label{eq:l=p-congruence}
    f_{r,m,k} \equiv_{\ell^{s+1-t}} f_{r,m,k+n},
    \end{equation}
    from which we obtain that $U(k+2) \equiv_{\ell^s} U(k+2+n)$ by the argument from Lemma \ref{lem:U_periodic}.
    By definition of $f_{r,m,k}$, we have for all $k \geq 2\lfloor (s-1)/a \rfloor$,
    \[
    f_{r,m,k} = \sum_{\substack{0 \leq j \leq \lfloor k/2 \rfloor \\ j \equiv \lfloor k/2 \rfloor -r \pmod{m} }} \binom{k-j}{j} (-q)^{j} \equiv_{\ell^s} \sum_{\substack{0 \leq j \leq s-1 \\ j \equiv \lfloor k/2 \rfloor -r \pmod{m} }} \binom{k-j}{j} (-q)^{j}.
    \]
    Put $m=m_{\ell,s}$. Either $n/2$ is divisible by $m_{\ell,s}$ or $\ell^s = 2$; in both cases the above gives 
    \[
    f_{r,m,k+n} - f_{r,m,k} \equiv_{\ell^s} \sum_{\substack{0 \leq j \leq s-1 \\ j \equiv \lfloor k/2 \rfloor -r \pmod{m} }} \left( \binom{k-j+n}{j} - \binom{k-j}{j} \right) (-q)^{j},
    \]
    so it suffices to show that
    \[
    \binom{k-j}{j} \equiv_{\ell^{s-j}} \binom{k-j+n}{j}.
    \]
    Since $\ell^{s-1} \mid n$, this follows from Lemma~\ref{lem:lucas}, using that $j \geq 1 + \lfloor \log_\ell(j)\rfloor$ for all $j \geq 1$. 

    Finally, $\mathrm{eis}_{H,k}(q)$ is a sum of contributions of the form $(\pm 1)^k$, so depends only on the parity of~$k$. The only case when $n$ is odd is when $\ell^s = 2$. But then $1 \equiv_{\ell^s} -1$, so $\mathrm{eis}_{H,k}(q)\equiv_{\ell^s} \mathrm{eis}_{H,k+n}(q)$ for all $k\geq 0$.
\end{proof}

\subsection{The case when $H$ is not representable and \texorpdfstring{$\ell=2$ or $3$}{ell23}} \label{sec-notrepresentable}

\begin{proof}[Proof of Theorem~\ref{thm:trace_periodicity} when $H$ is not representable and $\ell=2$ or $3$.]
Define $N(k+2)$ and $U(k+2)$ as in Lemma~\ref{lem:trace_formula_modulo}. 
Say that $\ell=3$, put $\tilde s=s+\nu_3(H,q)$ and replace $s$ by~$\tilde s$ in the definition of~$n$. Following the proof of Lemma~\ref{lem:N_periodic} we see that
    \[
     \biggl(\binom{(k+\delta_k)/2+j}{2j+\delta_k}-\binom{(k+\delta_k)/2+j+n/2}{2j+\delta_k}
     (-q)^{n/2}\biggr) a_1(E)^{2j+\delta_k} \equiv_{\ell^{\tilde s}} 0
    \]
for all $0 \leq j \leq \lfloor (\tilde s-1-\delta_k)/2\rfloor$ and $[(E,\phi)] \in \mathcal N(\bF_q)$.
After dividing this by $\# \Aut(E,\phi)$ and summing over all $[(E,\phi)] \in \mathcal N(\bF_q)$ it follows that 
\[
N(k+2)-N(k+2+n) \equiv_{\ell^{s}} 0. 
\]
Write $a_1(E)^{j} \equiv_{\ell^s} \sum_{t=1}^s b_{j,t}(E) \ell^{t-1}$  with $0 \leq b_{j,t} < \ell$, for any $[(E,\phi)] \in \mathcal U(\bF_q)$. Following the proof of Lemma~\ref{lem:U_periodic} we see that 
\[
\sum_{r=0}^{m_{\ell,\tilde s}-1} (f_{r,m_{\ell, \tilde s},k}- f_{r,m_{\ell, \tilde s},k+n}) \,a_1(E)^{2r+\delta_k}     \equiv_{\ell^{\tilde s}} \sum_{t=1}^{\tilde s} \ell^{t-1} \sum_{r=0}^{m_{\ell,t}-1} (f_{r,m_{\ell,t},k}-f_{r,m_{\ell,t},k+n}) \,b_{2r+\delta_k,t},
\equiv_{\ell^{\tilde s}} 0.
\]
After dividing this by $\# \Aut(E,\phi)$ and summing over all $[(E,\phi)] \in \mathcal N(\bF_q)$ it follows that 
\[
U(k+2)-U(k+2+n) \equiv_{\ell^{s}} 0.
\]
The statement in Theorem~\ref{thm:trace_periodicity} then follows in this case by Theorem~\ref{theorem-trace-Sk}.
The proof for $\ell=2$ is analogous, but uses Lemma~\ref{lem:U_periodic_even} instead of~\ref{lem:U_periodic}. 

If $\ell=2$ and $-\mathrm{Id}\in H$, then $-1 \in \Aut_{\F_q}(E,\phi)$ for any $(E,\phi) \in \mathcal Y_H(\bF_q)$. 
Twisting $(E,\phi) \in \mathcal Y_H(\bF_q)$ by the automorphism $-1$ gives an element  $(E^{(-1)},\phi) \in \mathcal Y_H(\bF_q)$ with $a_1(E^{(-1)})=-a_1(E)$. If we fix $t \neq 0$ then sending $E$ to $E^{(-1)}$ induces a bijection between the elements $[(E,\phi)] \in [\mathcal Y_H(\bF_q)]$ with $a_1(E)=t$ and the elements $[(E',\phi')] \in [\mathcal Y_H(\bF_q)]$ such that $a_1(E')=-t$. 
This gives an extra divisibility by~$2$ when going through the arguments above. 

Finally, in the case $\ell = p$, the arguments from Section~\ref{sec:ell=p} work also in characteristic~$2$ (respectively~$3$), with $s$ replaced by $s+\nu_2(H,q)$ or $s + \nu_2(H,q)-1$ if $-\mathrm{Id} \in H$ (respectively $s+\nu_3(H,q)$).
\end{proof}

\section{Proof of the main theorem for Drinfeld modular forms}\label{sec:ff-proof}
In this section we prove Theorem~\ref{thm:trace_periodicity_ff} for Drinfeld modular forms. The existence of a Frobenius map on~$A$ makes many statements easier, while types complicate matters.

\begin{notation} Throughout this section, we use the following notation.
    \begin{itemize}
    \item Fix $q$, a power of a prime~$p$. 
    \item Put $A = \mathcal O_C(C \setminus \{\infty\})$ for a smooth geometrically connected curve $C/\bF_q$, where $\infty$ is a point of degree~1. 
    \item Fix a compact open subgroup ${\mathcal K}$ of level~$\mathfrak n \trianglelefteq A$. 
     Fix non-zero prime ideals $\mathfrak p \nmid \mathfrak n$ and $\mathfrak l$ of~$A$ and an integer $s \geq 1$.
        \item Write $\tilde{s} := \lceil \log_p(s) \rceil$. Put $m_{\l,1} = |\l|-1$ if $p=2$ and if $(p,s) \neq (2,1)$ then write $m_{\l,s} := p^{\tilde{s}} (|\l|-1)/2$.
        \item Let $n \geq 1$ such that $\mathfrak p^n$ is principal. Denote by $\wp_n$ the unique monic generator of~$\mathfrak p^n$.
        \item Let
        \[
    \mathcal N(\bF_{\mathfrak p^n}) := \{ [(\phi,\lambda)] \in [\mathfrak M_{{\mathcal K},\mathfrak p}(\bF_{\mathfrak p^n})] \, | \, a_{\phi} \equiv_\l 0 \}, 
\qquad 
    \mathcal U(\bF_{\mathfrak p^n}) := [\mathfrak M_{{\mathcal K},\mathfrak p}(\bF_{\mathfrak p^n})] \setminus \mathcal N(\bF_{\mathfrak p^n}).
\]
        \item For any triple $(r,m,k) \in \mathbb Z^3$ with $k \geq 0$ and $m\geq 1$ and any $b \in \mathbb F_q^\times$, write
        \[
        g^{(b)}_{r,m,k} = \sum_{\substack{0 \leq j \leq \lfloor k/2 \rfloor \\ j \equiv r \pmod{m}}} \binom{k-j}{j} (-b\wp_n)^j.
        \]
\item 
Let $\delta_k := 0$ if $k$ is even and $\delta_k := 1$ if $k$ is odd.
    \end{itemize}
\end{notation}

\begin{lemma} \label{lem-exp}
    We have 
    \[
    \exp\bigl((A/\l^sA)^\times \bigr) = p^{\tilde{s}}(|\l|-1).
    \]
\end{lemma}

\begin{proof}
    Let $f \in (A/\l^sA)^\times$. Localization at $\l$ gives a ring isomorphism $A/\l^sA \cong A_{\l}/\l^sA_{\l} \cong \mathbb F_{\l}[\![\pi]\!]/(\pi^s)$, where $\pi$ denotes a local uniformizer.
    Since $\bF_{\l}^\times$ is cyclic of order $|\l|-1$, we may write $f^{|\l|-1} = 1 + a\pi$ under this isomorphism. Since the Frobenius map is additive, we see that $(f^{|\l|-1})^{p^n} \equiv 1 \pmod{\pi^s}$ whenever $p^n \geq s$, which shows that $\exp((A/\l^sA)^\times)$ divides $p^{\tilde{s}}(|\l|-1)$. In fact equality holds: if $\zeta$ is a generator of $\bF_{\l}^\times$, then the element $(\zeta + \pi)^n \not \equiv_{\pi^s} 1$ for any proper divisor $n$ of $p^{\tilde{s}}(|\l|-1)$. 
\end{proof}

 Theorem~\ref{thm:dmf_traceform_general} yields the following expression for traces of Hecke operators mod $\l^s$. It is proved
 in the same way as Lemma~\ref{lem:trace_formula_modulo}.

\begin{lemma}\label{lem:trace_formula_modulo_ff}
For any maximal ideal $\l \trianglelefteq A$, $s \geq 1$ and any $k \geq s-1$, we have
\[
\Tr(\T_{\mathfrak p}^n \, | \, \S_{k+2,l}({\mathcal K})) \equiv_{\l^s} -N(k+2,l) - U(k+2,l),
\]
where
\begin{align*} 
N(k+2,l) &:= \sum_{[(\phi,\lambda)] \in \mathcal N(\bF_{\mathfrak p^n})} \sum_{j=0}^{\lfloor \frac{s-1-\delta_k}{2} \rfloor} \binom{(k+\delta_k)/2+j}{2j+\delta_k} (-b_{\phi}\wp_n)^{(k-\delta_k)/2-j} \frac{a_{\phi}^{2j+\delta_k}b_{\phi}^{l-k-1}}{\#\Aut_{\mathbb F_{\mathfrak p^n}}(\phi,\lambda)};\\
    U(k+2,l) &:= \sum_{[(\phi,\lambda)] \in \mathcal U(\bF_{\mathfrak p^n})} \sum_{r=0}^{m_{\l,s}-1} g^{(b_{\phi})}_{\lfloor k/2\rfloor - r,m_{\l,s},k} \frac{a_{\phi}^{2r+\delta_k} b_{\phi}^{l-k-1}}{\#\Aut_{\mathbb F_{\mathfrak p^n}}(\phi,\lambda)}.
\end{align*}
\end{lemma}

\subsection{The case of $\l \neq \mathfrak{p}$}

\begin{lemma}
    Assume $\mathfrak l \neq \mathfrak p$. For $k \geq s-1$, we have $N(k+2,l) \equiv_{\l^s} N(k+2+n_N,l)$ where
    \[
    n_N = \begin{cases}
        2m_{\mathfrak l,s} = p^{\tilde{s}}(|\l|-1) & \text{if } p \neq 2 \text{ and } \lambda \wp_n \text{ is a square mod }\l \text{ for all } \lambda \in \mathbb F_q^\times;\\
        2p^{\tilde{s}}(|\l|-1) & \text{otherwise.}
    \end{cases}
    \] 
\end{lemma}

\begin{proof}
    Note that $\delta_k$ and $b_{\phi}^{-k}$ are invariant under $k \mapsto k + n_N$ since $\text{lcm}(2,q-1) \mid n_N$. By Lemma~\ref{lem-exp} and the assumptions,
    the same is true for $(-b_{\phi}\wp_n)^{k/2} \pmod{\l^s}$. It remains to show the periodicity of the binomial coefficients, which are computed in~$\bF_p$. For $s=1$, the sum in $N(k+2,l)$ is empty if $\delta_k = 1$, and the binomial coefficient is always 1 if $\delta_k = 0$.
    For $s \geq 2$, by Lemma~\ref{lem:lucas}, 
    the binomial coefficients occurring in the sum are invariant under $k \mapsto k + 2p^{1 + \lfloor \log_p( s-1 ) \rfloor}$. Since
    $\tilde{s} = \lceil \log_p(s) \rceil \geq 1 + \lfloor \log_p( s-1 ) \rfloor$, we find that the binomial coefficients are indeed invariant under $k \mapsto k + n_N$.
\end{proof}

\begin{lemma}\label{lem:series_rational_ff}
    Fix $m \geq 1$, $r \in \mathbb Z$, and $b \in \mathbb F_q^\times$. Then we have 
    \[
    \sum_{k \geq 0} g^{(b)}_{r,m,k} x^k = \frac{h(x)}{(1-x)^m - (b\wp_n x^2)^m}
    \]
    for some $h(x) \in A[x]$ of degree at most $2m-2$.
\end{lemma}

\begin{proof}
    Replace $q$ with a formal variable $t$ in the proof of Lemma~\ref{lem:series_rational}. Starting from $h_0(t,x) = (1-x+tx^2)^{-1} \in \mathbb Z[t][\![x]\!]$ and for $\zeta$ a complex primitive $m$-th root of unity, we obtain the equality
    
    \[
    \frac{1}{m} \sum_{i=1}^m \zeta^{ir}h_0(\zeta^{-i}t,x) = \sum_{k \geq 0} \Bigg{(} \sum_{\substack{0 \leq j \leq k/2 \\ j \equiv r \pmod{m}}} \binom{k-j}{j} (-t)^j \Bigg{)}x^k = \frac{h(x)}{(1-x)^m - (tx^2)^m} 
    \]
    for some $h(x) \in \mathbb Z[t][x]$ of degree at most $2m-2$. Reducing both sides of the equation mod~$p$ and substituting $t = b\wp_n$ completes the proof.
\end{proof}

\begin{lemma}\label{lem:div_upgrade_ff}
    Let $f \in A[x]$ and $t \geq 0$. Suppose that $f$ divides $x^n - 1$ mod $\l$. Then $f^{p^t}$ divides $x^{p^tn}-1$ mod $\l^{p^t}$.
\end{lemma}

\begin{proof}
    This follows by applying Frobenius $t$ times.
\end{proof}

\begin{lemma}\label{lem:n_U_ff}
    Let $s \geq 1$. Then for any $r \in \mathbb Z$ and any $k \geq 0$, we have
    \[
    g^{(b)}_{\lfloor k/2 \rfloor - r,m_{\l,s},k} \equiv_{\l^{s}} g^{(b)}_{\lfloor (k+n_U)/2 \rfloor - r,m_{\l,s},k+n_U}
    \]
    for all $b \in \mathbb F_q^\times$, where
    \[
    n_U = \begin{cases}
        \text{lcm}(2,p^{\tilde{s}}(|\l|^2-1)) & \text{if } p=2;\\
        p^{\tilde{s}}(|\l|^2-1) & \text{if } p>2,\ \deg(\l) \equiv_2 0 \text{ and } \left( \frac{\wp_n}{\l} \right) = -1;\\
        p^{\tilde{s}+1}(|\l|^2-1)/2 & \text{if }p>2,\ \deg(\l) \equiv_2 0 \text{ and } \left( \frac{\wp_n}{\l} \right) = 1;\\
        p^{\tilde{s}+1}(|\l|^2-1) & \text{otherwise.}
    \end{cases}
    \]
\end{lemma}

\begin{proof} 
    Suppose first that $p \neq 2$. Write $m' = (|\l|-1)/2$ so that $m_{\l,s} = p^{\tilde{s}} m'$. Then $m'$ is coprime to $p$ and so $v_p(m_{\l,s}) = \tilde{s}$. For any $b \in \bF_q^\times$, we have $b^{m'} \in \{\pm 1\}$. Let $\delta(b) := (1-b^{m'})/2$. Let $\alpha \in A$ be any lift of a primitive $2m'$-th root of unity mod~$\l$, i.e., $\alpha \pmod{\l}$ generates $\bF_{\l}^\times$. By Lemma~\ref{lem:rootofunity}, we then have
    \[
    (1-x)^{m'} - (b\wp_nx^2)^{m'} \equiv_{\l} \prod_{i=1}^{m'} (1 - x - \alpha^{2i+\delta(b)} \wp_n x^2).
    \]
    The factors of this product are clearly pairwise coprime. We distinguish two cases.
    
    If $\deg(\l)$ is even, then $\delta(b) = 0$ for all $b$, and the product contains a square if and only if $\wp_n$ is a square mod~$\l$. In this case, each root of a factor $1-x-\alpha^{2i}\wp_nx^2$ is a square in $\bF_{\l^2}$, by the same argument given in the proof of Lemma~\ref{lem:denom_divisible}.
    
    If $\deg(\l)$ is odd, then both $\delta(b) = 0$ and $\delta(b) = 1$ occur. In one of the cases the product will contain a square.
    
    Since each irreducible factor in the product is of degree at most~2, we find that in $\bF_{\l}[x]$,
    \[
    (1-x)^{m'} - (b\wp_nx^2)^{m'} \text{ divides } \begin{cases}
        x^{|\l|^2-1} - 1 & \text{if } \deg(\l) \equiv_2 0 \text{ and } \left( \frac{\wp_n}{\l} \right) = -1; \\
        x^{p(|\l|^2-1)/2} -1 & \text{if } \deg(\l) \equiv_2 0 \text{ and } \left( \frac{\wp_n}{\l} \right) = 1; \\
        x^{p(|\l|^2-1)} - 1 &  \text{if } \deg(\l) \equiv_2 1.
    \end{cases}
    \]
    Applying Lemma~\ref{lem:div_upgrade_ff} with $t = \tilde{s}$ together with Lemmas~\ref{lem:series_rational_ff} and~\ref{lem:periodic_criterion} gives
    \[
    g^{(b)}_{r,m_{\l,s},k} \equiv_{\l^s} g^{(b)}_{r,m_{\l,s},k+n_U},
    \]
    for any $k \geq 0$, $b \in \mathbb F_q^\times$ and any $r$.  Since $n_U$ is even and $n_U/2 \equiv 0 \pmod{m_{\l,s}}$, this implies the statement of the lemma (when $p \neq 2$).

    If $p=2$ then $m_{\l,1} = |\l|-1$ and $m_{\l,s} = p^{\tilde{s}-1}m_{\l,1}$ if $s \geq 2$. By Lemma~\ref{lem:series_rational_ff}, the denominator we are interested in is (the $p^{\tilde{s}-1}$-th power of) $d(x) := (1-x)^{|\l|-1}-(b\wp_n x^2)^{|\l|-1}$.
    If $\alpha \in A$ is a lift of a primitive element mod $\l$, then $d(x) \equiv_{\l} \prod_i (1-x+\alpha^i \wp_n x^2)$, which is a product of pairwise coprime polynomials, none of which are square. Thus $d(x)$ divides $x^{|\l|^2-1}-1$ mod $\l$, which implies $g^{(b)}_{r,m_{\l,s},k} \equiv_{\l^s} g^{(b)}_{r,m_{\l,s},k+p^{\tilde{s}}(|\l|^2-1)}$ in the usual way.
    Since $p^{\tilde{s}}(|\l|^2-1)$ divides $n_U$ and $n_U/2 \equiv_{m_{\l,s}} 0$, we are done.
\end{proof}

\begin{corollary}\label{cor:n_U_ff}
    We have $U(k+2,l) \equiv_{\l^s} U(k+2+n_U,l)$.
\end{corollary}

\begin{proof}
    Since $n_U$ is even, we have $\delta_k = \delta_{k+n_U}$. Moreover, $b_{\phi}^{n_U} = 1$ since $n_U$ is divisible by $q-1$. Hence the statement follows immediately from Lemmas~\ref{lem:trace_formula_modulo_ff} and~\ref{lem:n_U_ff}.
\end{proof}

We now conclude the periodicity of the traces.

\begin{proof}[Proof of Theorem~\ref{thm:trace_periodicity_ff} if $\l \neq \mathfrak p$]
    The period of the trace is $L := \text{lcm}(n_N,n_U)$. If $p=2$, then $L = 2p^{\tilde{s}}(|\l|^2-1) = p^{\tilde{s}+1}(|\l|^2-1)$. If $p \neq 2$, note that if $\wp_n$ is a square mod~$\l$ and $\deg(\l)$ is even, then $\lambda \wp_n$ is a square mod~$\l$ for all $\lambda \in \bF_q^\times$. Hence $n_N \mid n_U$ in all cases, so $L = n_U$.
\end{proof}

\subsection{The case of $\l = \mathfrak p$}
Suppose now that $\l = \mathfrak p$. Then we have the following periodicity of the traces.

\begin{theorem}
    For all $k \geq k_0$, we have
    \[
    \Tr(\T_{\mathfrak p}^n \, | \, \S_{k+2,l}({\mathcal K})) \equiv_{\mathfrak p^s} \Tr(\T_{\mathfrak p}^n \, | \, \S_{k+2+p^{\tilde{s}}(|\l|-1),l}({\mathcal K})),
    \]
    where $k_0 = 2s-1$ if $n = 1$ and $k_0 = s$ if $n \geq 2$.
\end{theorem}

\begin{proof}
    Clearly $N(k+2,l) \equiv_{\mathfrak p^s} 0$ for all $k \geq 2s-1$ if $n=1$ and $k \geq s$ if $n \geq 2$.

    We now establish the period of $U(k+2,l)$. For all $k \geq 2\lfloor (s-1)/n \rfloor$ and $b \in \mathbb F_q^\times$, we have
    \[
    g^{(b)}_{r,m,k} = \sum_{\substack{0 \leq j \leq \lfloor k/2 \rfloor \\ j \equiv r \pmod{m}}} \binom{k-j}{j} (-b\wp_n)^j \equiv_{\mathfrak p^s} \sum_{\substack{0 \leq j \leq s-1 \\ j \equiv r \pmod{m}}} \binom{k-j}{j} (-b\wp_n)^j.
    \]
    By Lemma~\ref{lem:lucas}, 
    all binomial coefficients which occur are invariant under $k \mapsto k + p^{\tilde{s}}$, which implies that 
    \[
    g^{(b)}_{r,m,k} \equiv_{\mathfrak p^s} g^{(b)}_{r,m,k+p^{\tilde s}}
    \]
    for all $k \geq 2\lfloor (s-1)/n \rfloor$, $m \geq 1$, $b \in \mathbb F_q^\times$ and $r \in \mathbb Z$.

    Now let $N = p^{\tilde{s}}(|\l|-1)$ and $b \in \mathbb F_q^\times$. If $p^s \neq 2$ then $N = 2m_{\mathfrak p,s}$, and by the above,
    \[
    g^{(b)}_{\lfloor k/2 \rfloor - r,m_{\mathfrak p,s},k} \equiv_{\l^s} g^{(b)}_{\lfloor (k+N)/2 \rfloor - r,m_{\mathfrak p,s},k+N}.
    \]
    Moreover, since $q-1$ divides $N$ we have $b^N = 1$ and since $N$ is even we have $\delta_{k+N} = \delta_k$ for all $k$. Hence it follows from Lemma~\ref{lem:trace_formula_modulo_ff} that $U(k+2,l) \equiv_{\mathfrak p^s} U(k+2+N,l)$. Finally if $p^s = 2$ then one sees directly that
    \[
    U(k+2,l) \equiv_{\l} \sum_{[(\phi,\lambda)] \in \mathcal U(\mathbb F_{\mathfrak p^n})} \frac{a_{\phi}^{k}b_{\phi}^{l-k-1}}{\# \Aut_{\mathbb F_{\mathfrak p^n}}(\phi,\lambda)},
    \]
    which is also invariant under $k \mapsto k + |\l|-1$.
\end{proof}

\subsection{Traces modulo \texorpdfstring{$\infty$}{infty}}
Let $\mathcal O_\infty = \{x \in K \ | \ v_\infty(x) \geq 0\}$ be the local ring of $\infty$-integral elements in $K$. Fix a generator $\pi_\infty$ of its maximal ideal and recall that we assume $\mathcal O_\infty / (\pi_\infty) = \mathbb F_q$.

Define
\[
\Tr_\infty(\T^n_{\mathfrak p} \, | \, \S_{k+2,l}({\mathcal K})) := \frac{1}{(-\wp_n)^{\lceil k/2 \rceil}} \Tr(\T^n_{\mathfrak p} \, | \, \S_{k+2,l}({\mathcal K})).
\]
Then $\Tr_\infty(\T^n_{\mathfrak p} \, | \, \S_{k+2,l}({\mathcal K})) \in \mathcal O_\infty$; this follows from Theorem~\ref{thm:dmf_traceform_general} and the Riemann hypothesis for Drinfeld modules over finite fields, which says that $-2v_\infty(a_{\phi}) \leq -v_\infty(\wp_n)$, see \cite[Theorem~5.1]{gek_finite}. If $A = \bF_q[T]$, then $v_\infty = -\deg_T$ and the strong Ramanujan bound \cite[Conjecture 4.3]{devries_traces} is equivalent to the statement
\[
\Tr_\infty(\T^n_{\mathfrak p} \, | \, \S_{k+2,l}(\mathrm{GL}_2(A))) \equiv 0 \pmod{\pi_\infty^{n\deg(\mathfrak p)(q-1)/2}} \qquad \text{for all } k,l \in \mathbb Z.
\]
Working in $\mathcal O_\infty$ instead of $A$, we can argue as before to show that $\Tr_\infty$ is periodic in the weight modulo powers of $\pi_\infty$. We sketch the argument here.
Define
\begin{align*}
\mathcal N_\infty(\bF_{\mathfrak p^n}) &= \left\{ [(\phi,\lambda)] \in [\mathfrak M_{{\mathcal K},\mathfrak p}( \bF_{\mathfrak p^n} )] \, \bigg{|} \, -v_\infty(a_{\phi}) < \frac{-v_\infty(\wp_n)}{2} \right \};\\
\mathcal U_\infty(\bF_{\mathfrak p^n}) &= [\mathfrak M_{{\mathcal K},\mathfrak p}( \bF_{\mathfrak p^n} )] \setminus \mathcal N_\infty(\bF_{\mathfrak p^n}).
\end{align*}
Then we have, for all $k \geq s-1$,
\[
\Tr_\infty(\T_{\mathfrak p}^n \, | \, \S_{k+2,l}({\mathcal K})) \equiv_{\pi_\infty^s} -N_\infty(k+2,l) - U_\infty(k+2,l),
\]
where 
\begin{align*}
    N_\infty(k+2,l) &:= \sum_{[(\phi,\lambda)] \in \mathcal N_\infty(\bF_{\mathfrak p^n})} \sum_{j=0}^{\lfloor \frac{s-1-\delta_k}{2}\rfloor} \binom{(k+\delta_k)/2+j}{2j+\delta_k} \left( \frac{a_{\phi}^{2j+\delta_k}}{-\wp_n^{2j+\delta_k}}\right)\frac{b_{\phi}^{l-j-(k+\delta_k)/2-1}}{\# \Aut_{\mathbb F_{\mathfrak p^n}}(\phi,\lambda)};\\
    U_\infty(k+2,l) &:= \sum_{[(\phi,\lambda)] \in \mathcal U_\infty(\bF_{\mathfrak p^n})} \sum_{r=0}^{m_{\infty,s}-1} h^{(b_{\phi})}_{\lfloor k/2 \rfloor - r,m_{\infty,s},k} \left( \frac{a_{\phi}^{2j+\delta_k}}{-\wp_n^{2j+\delta_k}}\right) \frac{b_{\phi}^{l-k-1}}{\# \Aut_{\mathbb F_{\mathfrak p^n}}(\phi,\lambda)}.
\end{align*}

Here we write $m_{\infty,s} := p^{\tilde{s}}(q-1)/2$ (unless $p^s = 2$, in which case $m_{\infty,s} := q-1$), and
\[
h^{(b)}_{r,m,k} := \sum_{\substack{0 \leq j \leq \lfloor k/2 \rfloor \\ j \equiv r \pmod{m}}} \binom{k-j}{j} b^j.
\]

Using the same arguments as before, we obtain:

\begin{proposition} For all $k \geq s-1$, 
    \[
    \Tr_\infty(\T_{\mathfrak p}^n \, | \, \S_{k+2,l}(\mathcal K)) \equiv_{\pi_\infty^s} 
    \Tr_\infty(\T_{\mathfrak p}^n \, | \, \S_{k+2+p^{1+\tilde{s}}(q^2-1),l}(\mathcal K)).
    \]
\end{proposition} 

In particular, this reduces the verification of the strong Ramanujan bound for each $\mathfrak p^n$ to a finite computation:

\begin{corollary}
    Fix a maximal ideal $\mathfrak p \trianglelefteq \bF_q[T]$ and an integer $n \geq 1$. Let $s = n\deg(\mathfrak p)(q-1)/2$ and $\tilde{s} = \lceil \log_p(s)\rceil$. Then the strong Ramanujan bound holds for $\T_{\mathfrak p}^n$ if it holds on the spaces $\S_{k+2,l}$ with $1 \leq l \leq q-1$ and $0 \leq k < p^{1 + \tilde{s}}(q^2-1) + s$.
\end{corollary}

\section{Examples}\label{sec:examples}

\subsection{Elliptic modular forms of level $1$ for small $\ell^s$}
\label{seq-congruence}
For any $i \geq 1$ put $h_i(x)=x^{k-i} \, \prod_{j=1}^i(x-j)=x^k+\sum_{j=1}^i c_{i,j} x^{k-j} \in \mathbb Z[x]$ for some $c_{i,j} \in \mathbb Z$. For any integer $x_0$ and prime $\ell$, we see that 
\[
\prod_{j=1}^i (x-j) \equiv_{\ell^{t_i}} 0 \;\; \text{where} \;\; t_i= v_{\ell}(i!).
\]
Let $N\geq 1$ and let $H$ be a subgroup of $\mathrm{GL}_2(\mathbb Z/N\mathbb Z)$. It follows that if either $H$ is representable or $\ell \geq 5$ that 
\begin{equation} \label{eq-faculty}
\sum_{[(E,\phi)] \in [\mathcal{Y}_{H}(\bF_q)]} \frac{h_i(a_1(E))}{\# \Aut_{\bF_q}(E,\phi)}\equiv_{\ell^{t_i}} [a_1^k] +\sum_{j=1}^i c_{i,j} \, [a_1^{k-j}] \equiv_{\ell^{t_i}} 0.
\end{equation}
Following the proof in Section~\ref{sec-notrepresentable}, we see that if $H$ is not representable and $\ell=2$ or $3$, then Equation~\eqref{eq-faculty} holds but with $t_i$ replaced by $t_i+\nu_\ell(H,q)$ and that if $\ell=2$ and $-\mathrm{Id}\in H$ then $t_i$ can instead be replaced by $t_i-1+\nu_2(H,q)$.

For fixed $i$ and $\ell$, knowing $[a_1^j]$ modulo $\ell^{t_i}$ for all $j \leq i$, we can use Equation~\eqref{eq-faculty} to compute $[a_1^k]$ modulo $\ell^{t_i}$ for all $k$. 

For $\Gamma=\Gamma(1)$, $\mathrm{eis}_{\Gamma(1),k}(q)=1$ for all $q$ and even $k \geq 0$. We can use the fact that $S_{k}(\Gamma(1))=\{0\}$ for all $k<12$ together with Theorem~\ref{theorem-trace-Sk} to compute $[a_1^{k}]$ for all $q$ and all $k <10$.
We find that for any $q$ we have,  
\begin{align*}
[a_1^0]&=q, \\
[a_1^2]&=q^2-1 \\
[a_1^4]&=2q^3-3q-1\\
[a_1^6]&=5q^4-9q^2-5q-1 \\
[a_1^8]&=14q^5-28q^3-20q^2-7q-1.
\end{align*}

Using the results of Section~\ref{seq-congruence} together with \cite[Theorem~10.1]{silverman}, we find that 
 \[
 [a_1^{10+2i}] \equiv_{\ell^{r_{\ell}(p)}} 870\cdot [a_1^{8+2i}]+63273\cdot[a_1^{6+2i}]+723680\cdot[a_1^{4+2i}]+1026576\cdot[a_1^{2+2i}]
 \]
for any $q=p^r$ and $i \geq 0$, where $r_7(p)=1,r_5(p)=2,r_3(p)=3$ for any prime~$p$, and where $r_2(2)=6$ and $r_2(p)=7$ for any $p \neq 2$ (and $r_{\ell}(p) = 0$ for any $\ell \geq 11$).
Using this we can recursively compute $[a_1^k]$ modulo $\ell^{r_{\ell}(p)}$ for any $k \geq 0$. 

 \begin{example}\label{ex:elliptic_wt28}
 Using the results above we computed the following examples: 
\begin{align*}
\Tr(F_q \, | \, S[\Gamma(1),28]) &\equiv_{5^2} 15q^{14}+15q^{13}+22q^{12}+19q^{10}+10q^9+19q^8+22q^7+7q^6+21q^5\\ &+10q^4+4q^3+10q^2+5q,\\
\Tr(F_q \, | \, S[\Gamma(1),28]) &\equiv_{3^3} 9q^{12}+18q^{10}+26q^9+21q^7+20q^6+6q^5+18q^4+3q^3+18q^2,\\
\Tr(F_q \, | \, S[\Gamma(1),28]) &\equiv_{2^{r_2(p)}} 92q^{14}+24q^{13}+26q^{12}+3q^{11}+6q^{10}+96q^9+7q^8+74q^7+108q^6\\&+24q^5+31q^4+37q^3+116q^2,
\end{align*} 
for any $q=p^r$.
Note that by virtue of Theorem~\ref{thm:trace_periodicity}, these congruences also hold in weights $28+kn$ for all $k \geq 1$, for a suitable~$n$. Thus, doing a similar computation for a finite number of weights gives polynomial expressions for traces modulo $\ell^s$ in all weights, for $\ell^s \in \{5^2,3^3,2^{r_2(p)}\}$.

For results with $\ell \leq 7$, $s=1$ and $\ell \neq p$, see  
\cite[Theorem~2a]{CFW}.
  \end{example}

\begin{example} For $\ell=11$ we are just outside the range mentioned above, and the issue will be that $a_1(E)^{\ell-1}\equiv_{\ell}0$ for any $[(E,\phi)] \in \mathcal N(\mathbb F_q)$.
    For any $q=p^r$ we find that 
    \[
    \Tr(F_q \, | \, S[\Gamma(1),12]) \equiv_{11} \# \mathcal N(\mathbb F_q)+9q^6+9q^4+2q^3+9q^2+q+10.
    \]
    Note that for any prime $p \geq 5$ we have that 
    \[
    \# \mathcal N(\mathbb F_p)=\frac{1}{2} \left(H(-4p)+\sum_{i=1}^{\lfloor 2\sqrt{p}/11\rfloor} H \left((11i)^2-4p \right)\right),
    \]
    where $H(\Delta)$ denotes the Kronecker class number of $\Delta$, see \cite[Theorem 4.6]{Schoof}. 
\end{example}

\subsection{Congruence equations over function fields}
\begin{example}
    When $A = \mathbb F_q[T]$, $\mathcal{K} = \mathrm{GL}_2(\hat{A})$, and $\l$ is a prime of degree 1, the traces of Hecke operators mod~$\l$ are known: for $\alpha \in \bF_q$ and $\mathfrak p \neq (T-\alpha)$ with monic generator~$\wp$, we have \cite[Cor.~3.3]{joshi-petrov}
    \[
    \Tr(\T_{\mathfrak p}^n \, | \, \S_{k,l}) \equiv \wp(\alpha)^{n(l-1)} \cdot \dim \S_{k,l} \pmod{(T-\alpha)}.
    \]
    Since $\dim \S_{k,l} = \lfloor \frac{k + (q-1-l)(q+1)}{q^2-1} \rfloor$ if $1 \leq l \leq q-1$ and $k \equiv_{q-1} 2l$, we see that the trace is periodic mod $\l$ with period $p(q^2-1)$. In particular, the period from Theorem~\ref{thm:trace_periodicity_ff} is optimal in this case.
\end{example}

\begin{example}
    In the setting of the previous example, the Drinfeld cusp form of lowest weight which does not possess an $A$-expansion is $gh^2 \in \S_{3q+1,2}(\mathrm{GL}_2(\hat{A}))$. One can use this to show that the elements $[c_{k,l}]$ from Theorem~\ref{thm:dmf_traceform_general} lie in $\bF_p \subset \bF_q[T]$ for $k < 3q-1$. 
    
    In particular, if $q=3$ and $\l$ is a prime of degree~1, we have $a^{k+6} \equiv_{\l^2} a^{k}$ for all $k \geq 2$ and all $a \in A$. It then follows from Theorem~\ref{thm:dmf_traceform_general} that for all $k$ and~$l$, $\Tr(\T_{\mathfrak p}^n \, | \, \S_{k+2,l}(\mathrm{GL}_2(\bF_3[T])))\pmod{\l^2}$ is a polynomial in $\wp^n$ over~$\bF_3$. In particular, this trace depends only on $\wp^n \pmod{\l^2}$, analogously to Example~\ref{ex:elliptic_wt28}.
    
    Consider for instance the case $\l = (T)$, so we consider traces of Hecke operators modulo~$T^2$. Assume $k$ is even (as otherwise $\S_{k,l} = 0$). Let $l \in \{0,1\}$ and write $d_{k,l} := \dim \S_{k,l} = \lfloor \frac{k+4l}{8} \rfloor$. Then we have for any $\mathfrak p$ with monic generator $\wp \neq T$ and any $n \geq 1$,
    \[
    \Tr(\T_{\mathfrak p}^n \, | \, \S_{k,l}) \equiv \wp^{n(1-l)} \left(a_{k,l} + (d_{k,l} - a_{k,l})\wp^{2n}\right) \pmod{T^2},
    \]
    where
    \[
    a_{k,l} = \begin{cases}
        -l -\left\lfloor \frac{k-4l}{24}\right\rfloor & \text{if } d_{k,l} \equiv_3 0;\\
        (l-1)k + l -\left\lfloor \frac{k-4l}{24}\right\rfloor & \text{if } d_{k,l} \equiv_3 1;\\
        (-1)^{l+1}k + l -\left\lfloor \frac{k-4l}{24}\right\rfloor & \text{if } d_{k,l} \equiv_3 2.
    \end{cases}
    \]
    Using Theorem~\ref{thm:trace_periodicity_ff}, this could be verified with a finite number of computations.
\end{example}

\subsection{Hecke polynomials}
Fix a subgroup $H \subseteq \mathrm{GL}_2(\mathbb Z/N\mathbb Z)$ and a prime $p \geq 5$ which does not divide~$N$. As before, write $S_k(H) = \bigoplus_{i=1}^d S_k(\Gamma_i)$, where the $\Gamma_i$ are the congruence subgroups such that $\mathcal Y_H(\mathbb C) \cong \coprod_{i=1}^d \Gamma_i \backslash \mathcal{H}$. Define for any $k \geq 0$,
\[
\mathfrak f_{k+2}(x) := \det(1-T(p)x \ | \ S_{k+2}(H)) \in \mathbb Z[x].
\]

\begin{proposition}\label{prop:pols}
    Suppose $H$ is representable. 
    Then $\f_{k}(x) \equiv_p \f_{k+p-1}(x)$ for all $k \geq 3$.
\end{proposition}

\begin{proof}
If $\mathcal{X}_H$ is a scheme, and $k \geq 1$, then the trace formula can equivalently be written as
\[
\det(1-F_px \, | \, S[H,k+2]) = \prod_{n \geq 1} \prod_{(E,\phi)/\mathbb F_{p^n}} \det(1-F_{p^n}x^n \, | \, \Sym^k(H^1(E \otimes \overline{\bF}_q, \mathbb Q_{\ell'})))^{-1},
\]
where the product is over the closed points $(E,\phi) \in |\mathcal{X}_H \otimes \mathbb F_p|$ of degree~$n$. If $(E,\phi)$ is a cusp, then 
\[
\det(1-F_{p^n}x^n \, | \, \Sym^k(H^1(E \otimes \overline{\bF}_q, \mathbb Q_{\ell'})))^{-1} = (1 - (\pm 1)^{k} x^n)^{-1},
\]
depending on whether $E$ is of split or non-split multiplicative type. In either case, since $p$ is odd, this factor is invariant under $k \mapsto k + p-1$.

If $(E,\phi) \in |\mathcal{Y}_H \otimes \mathbb F_p|$, we have
\[
\det(1-F_{p^n}x^n \, | \, \Sym^k(H^1(E\otimes \overline{\bF}_q, \mathbb Q_{\ell'})))^{-1} = \prod_{i=0}^k (1- \pi_E^i \bar{\pi}_E^{k-i}x^n)^{-1}, 
\]
where $X^2 - a_1(E) X + p^n = (X-\pi_E)(X-\bar{\pi}_E)$. In particular, $\pi_E \bar{\pi}_E = p^n$ and hence
\begin{align*}
\det(1-F_{p^n}x^n \, | \, \Sym^k(H^1(E\otimes \overline{\bF}_q, \mathbb Q_{\ell'})))^{-1} &\equiv_{p} (1-(\pi_E^k + \bar{\pi}_E^k)x^n)^{-1} \\
&\equiv_{p} (1 - (\pi_E + \bar{\pi}_E)^kx^n)^{-1} = (1-a_1(E)^kx^n)^{-1}.
\end{align*}
Since $a_1(E) \in \mathbb Z$, this is invariant mod $p$ under $k \mapsto k + p-1$, provided that $k \geq 1$.

Finally,
\[
\det(1-F_px \, | \, S[H,k+2]) = \prod_{i=1}^d (1 - a_p(f_i)x + p^{k-1}x^2) \equiv_p \det(1-T(p)x \, | \, S_{k+2}(H)). \qedhere
\]
\end{proof}

\begin{remark}
    If $H$ is not representable, it is in some cases still possible to deduce the conclusion of the proposition by applying Newton's identities to the traces of $T(p)$ on the spaces $W_k(H) = \tl{S}_{k+p-1}(H)/E_{p-1}\tl{S}_k(H)$, where $\tl{S}_k(H)$ is the reduction modulo~$p$ of $S_k(H)$ and $E_{p-1}$ denotes the Eisenstein series of weight~$p-1$, which satisfies $E_{p-1} \equiv_p 1$. Since Newton's identities involve denominators, one needs to assume that $p > \dim_{\mathbb F_p} W_k(H)$; this holds for instance if $H$ corresponds to $\Gamma_1(N)$ with $N \leq 3$
    (and $\Gamma_1(N)$ is representable for $N \geq 4$, see \cite[Cor.~2.7.3]{KatzMazur}).
\end{remark}

The same proof works for Drinfeld modular forms, which gives the following result. Note that $\mathcal K$ is representable if it is admissible in the sense of~\cite{bockle}.

\begin{proposition}
    Suppose $\mathcal K$ is representable. Then 
    \[
    \det(1-\T_{\mathfrak p}x \, | \, \S_{k,l}(\mathcal{K})) \equiv_{\mathfrak p} \det(1-\T_{\mathfrak p}x \, | \, \S_{k+|\mathfrak p|-1,l}(\mathcal K))
    \]
    for any $k \geq 3$ and $l \in \mathbb Z$. 
\end{proposition}

Define a \emph{$p$-adic slope} of $T(p)$ to be the normalized $p$-adic valuation of a root of $\f_{k}(x) \in \mathbb Z_p[x]$, and denote by $d_p(k,\alpha)$ its multiplicity. Similarly, for Drinfeld modular forms, denote by $d_{\mathfrak p}(k,l,\alpha)$ the multiplicity of the $\mathfrak p$-adic slope $\alpha$ on $\S_{k,l}(\mathcal{K})$. Retaining the representability assumption on the level, we deduce a dimension periodicity for slope~0.

\begin{corollary}\label{cor:GM}
    If $k \geq 3$, then $d(k,0) = d(k+p-1,0)$ and $d_{\mathfrak p}(k,l,0) = d_{\mathfrak p}(k+|\mathfrak p|-1,l,0)$.
\end{corollary}

\begin{proof}
By the theory of Newton polygons, $d_p(k,0)$ equals the degree of $\f_k(x) \pmod{p}$, and $d_{\mathfrak p}(k,l,0)$ equals the degree of $\det(1-\T_{\mathfrak p}x \, | \, \S_{k,l}(\mathcal K)) \pmod{\mathfrak p}$.
\end{proof}

In particular, with the exception of $k=2$, this recovers Hida's result for elliptic modular forms that the Gouv\^ea--Mazur conjecture holds for slope 0.

\begingroup
\sloppy
 \printbibliography
\endgroup

\end{document}